\documentclass[a4paper, reqno]{amsart}

\usepackage{cite}
\usepackage{amsmath}
\usepackage{amssymb}
\usepackage{amsthm}   %-------proof
\usepackage{empheq}   %-------to use \empheq
\usepackage{graphicx}
\usepackage{color}
\usepackage{appendix}
\usepackage{hyperref} 

\newtheorem{theorem}{Theorem}[section]
\newtheorem{lemma}[theorem]{Lemma}
\newtheorem{proposition}[theorem]{Proposition}
\newtheorem{corollary}[theorem]{Corollary}

\theoremstyle{definition}

\newtheorem{example}[theorem]{Example}

\theoremstyle{remark}
\newtheorem{remark}[theorem]{Remark}

\numberwithin{equation}{section}

\def\EE{{\mathcal{E}}}
\def\FF{{\mathcal{F}}}
\def\ss{{\mathtt{s}}}
\def\tt{{\mathtt{t}}}

\begin{document}

\title[Class of smooth functions]{Class of smooth functions in Dirichlet spaces}

%    Information for second author
\author{Patrick J. Fitzsimmons}
\address{Department of Mathematics, University of California, San Diego, La Jolla, California 92093-0112}
\email{pfitzsim@ucsd.edu}
%\thanks{Support information for the second author.}

%    Information for first author
\author{ Liping Li*}
%    Address of record for the research reported here
\address{Institute of Applied Mathematics, Academy of Mathematics and Systems Science, Chinese Academy of Sciences, Beijing 100190, China.}
%    Current address
%\curraddr{Department of Mathematics and Statistics,
%Case Western Reserve University, Cleveland, Ohio 43403}
\email{liping\_li@amss.ac.cn}
%    \thanks will become a 1st page footnote.
\thanks{*Corresponding author.}

%    General info

%\date{January 1, 2001 and, in revised form, June 22, 2001.}

%\dedicatory{This paper is dedicated to our advisors.}


\begin{abstract}
Given a regular Dirichlet form $(\EE,\FF)$ on a fixed domain $E$ of $\mathbb{R}^d$, we first indicate that the basic assumption $C_c^\infty(E)\subset \FF$ is equivalent to the fact that each coordinate function $f^i(x)=x_i$ locally belongs to $\FF$ for $1\leq i\leq d$. Our research starts from these two different viewpoints. On one hand, we shall explore when $C_c^\infty(E)$ is a special standard core of $\FF$ and give some useful characterizations. On the other hand, we shall describe the Fukushima's decompositions of $(\EE,\FF)$ with respect to the coordinates functions, especially discuss when their martingale part is a standard Brownian motion and what we can say about their zero energy part. Finally, when we put these two kinds of discussions together, an interesting class of stochastic differential equations are raised. They have uncountable solutions that do not depend on the initial condition. 
\end{abstract}

\maketitle

\tableofcontents

\section{Introduction}\label{SEC1}

Our interests are motivated by S. Orey's work \cite{O74}, in which a complete description on the diffusion processes was given to ensure that  they are absolutely continuous with respect to one-dimensional Brownian motion. More precisely, let $I$ be an open one-dimensional interval and $\mathbf{X}^1,\mathbf{X}^2$ two irreducible diffusion processes on $I$. We use the standard way via coordinate space to realize such diffusion processes: $\Omega$ is the class of all continuous functions from $[0,\infty)$ to $I$ and for any $\omega\in \Omega$, $X_t(\omega):=\omega(t)$. Let $\mathcal{F}_t$ be the $\sigma$-field generated by $\{X_s: s\leq t\}$ and $\mathcal{F}_\infty$ the least $\sigma$-field including all the $\FF_t$. Then $\mathbf{X}^i$ corresponds to a class of probability measures $(\mathbf{P}_i^x)_{x\in I}$ on $\Omega$ for $i=1,2$. In what follows, we say $\mathbf{X}^1$ is absolutely continuous with respect to $\mathbf{X}^2$ if $\mathbf{P}_1^x|_{\FF_t}\ll \mathbf{P}_2^x|_{\FF_t}$ for any $x\in I$ and $t\geq 0$, where ``$\ll$" stands for ``is absolutely continuous with respect to''. When $I=(-\infty,\infty)$ and $\mathbf{X}^2$ is a Brownian motion, S. Orey deduced that $\mathbf{X}^1$ admits the following decomposition:
\begin{equation}\label{EQ1XTX}
	X_t-x=B_t+\int_0^t b(X_s)ds,\quad \mathbf{P}_1^x\text{-a.s.},
\end{equation}
where $B=(B_t)_{t\geq 0}$ is a standard Brownian motion and $b$ is  some appropriate function. By introducing the techniques of Dirichlet forms, M. Fukushima in \cite{Fu82} extended the above result to multidimensional symmetric diffusions and also achieved similar decomposition to \eqref{EQ1XTX}. Then the general cases that $\mathbf{X}^1$ and $\mathbf{X}^2$ are both symmetric Markov processes (not only diffusion processes) on some general state spaces are considered by \cite{Os87, K76, F97, CFT04}. 

In S. Orey and M. Fukushima's literatures, the decomposition \eqref{EQ1XTX} admits an obvious fact: the coordinate function $f(x):=x,~\forall x\in \mathbb{R}$ locally belongs to the associated Dirichlet space $\FF$ of $\mathbf{X}^1$, in other words, $f\in \FF_\mathrm{loc}$. Hence \eqref{EQ1XTX} is nothing but the Fukushima's decomposition of $\mathbf{X}^1$ with respect to $f$. Particularly, its martingale part is a Brownian motion and zero energy part is of locally bounded variation. 

So some relevant questions arise naturally at this point. For example, except for S. Orey's absolute continuity, are there any other ways to get diffusion processes which satisfy the similar decompositions to \eqref{EQ1XTX}? But the word ``similar'' here is slightly crude. We may explain more.  At least, we want to assume the diffusion process and its associated Dirichlet space $\FF$ satisfy the basic property:
\begin{equation}\label{EQ1FFL}
	f\in \FF_\mathrm{loc}. 
\end{equation}
Then we have the following Fukushima's decomposition:
\begin{equation}\label{EQ1XXM}
	f(X_t)-f(X_0)=M^{[f]}_t+N^{[f]}_t,\quad t\geq 0,
\end{equation}
where $M^{[f]}$ is a locally martingale additive functional (MAF) of finite energy and $N^{[f]}$ is a locally continuous additive functional (CAF) of zero energy. It is possible to analyse $M^{[f]}$ and $N^{[f]}$ carefully, for instance, when $M^{[f]}$ is a Brownian motion, and what we can say about the zero energy part $N^{[f]}$. Actually we shall give some detailed considerations about these questions, such as in \S\ref{SEC4}. 

In the mean time, when we recheck the basic assumption \eqref{EQ1FFL}, it is very easy to find \eqref{EQ1FFL} is equivalent to
\begin{equation}\label{EQ1CIF}
	C_c^\infty(I)\subset \FF,
\end{equation}
where $C_c^\infty(I)$ is the class of all infinite continuous differentiable functions with compact supports on $I$. We refer this equivalence to \cite{BD59, Fu82}. However, for most of the classical examples 
in Dirichlet forms (say \cite[\S3.1]{FOT11}), the processes are generated by the special standard core $C_c^\infty$, in other words, the smooth function class $C_c^\infty$ is dense in the Dirichlet space $\FF$ with respect to its natural norm. Very few articles concern about whether $C_c^\infty$ would be a core of a fixed Dirichlet space or not. (As we know, some similar works were done in \cite{C92, Ku02, SU12} from the aspect of reflected Dirichlet spaces). Indeed, once we prove the closable property of $C_c^\infty$, its smallest closed extension will naturally be a regular Dirichlet form and hence correspond to a good Markov process. But in many situations, we always need to start some work with a fixed but unclear Dirichlet space $\FF$. So we think it is worth discussing the denseness of $C_c^\infty$ in $\FF$ under the basic condition \eqref{EQ1CIF}. In the beginning, this work is supposed to be independent of the research on Fukushima's decomposition \eqref{EQ1XXM}, but soon we find some deep and interesting connections between them.

Let us briefly introduce the structure of this paper. In \S\ref{SEC2}, we shall present our main results and their applications, which attract our interests. The focus is to characterize the basic assumption \eqref{EQ1FFL}, and answer when $C_c^\infty$ could be a special standard core of $(\EE,\FF)$. Particularly, we shall raise a class of stochastic differential equations that have uncountable solutions in \S\ref{SEC24}. The results of \S\ref{SEC21} and \S\ref{SEC22} will be proved in \S\ref{SEC3}. In \S\ref{SEC4}, we shall describe the Fukushima's decomposition \eqref{EQ1XXM}, especially about when $M^{[f]}$ is a Brownian motion, and whether $N^{[f]}$ is of bounded variation or not. The descriptions in this section provide a possible method to prove the conclusions in \S\ref{SEC23} and \S\ref{SEC24}. Finally, in \S\ref{SEC5}, we shall give several examples to explain or support the results in previous sections. 

Some notations should be noted first. Fix a domain $E$ of $\mathbb{R}^d$, the function classes $C(E),C^1(E)$ and $C^\infty(E)$ are the sets of all continuous functions, one order continuous differentiable functions, and infinite continuous differentiable functions on $E$ respectively. For any Radon measure $\mu$ on $E$, the Hilbert space of all $\mu$-square integrable functions on $E$ is denoted by $L^2(E,\mu)$ or $L^2(\mu)$. Its norm and inner product are denoted by $\|\cdot \|_\mu$ and $(\cdot, \cdot)_\mu$. Particularly, if $\mu$ is the Lebesgue measure (we usually denote it by $\lambda_E$ or $dx$), $L^2(E,\mu)$ will be written as $L^2(E)$ or $L^2$ in abbreviation. The similar notations are also provided for integrable functions in the usual way.  For any function class $\Theta$, the subclass of all the functions locally (resp. with compact support, bounded) in $\Theta$ will denoted by $\Theta_{\text{loc}}$ (resp. $\Theta_c,\Theta_b$). For a mapping $p:E_1\rightarrow E_2$, where $E_1,E_2$ are two subsets of $\mathbb{R}^d$, define 
\[
	p(E_1):=\{p(x)\in E_2:x\in E_1\}.
\]
For any function $p$ on $\mathbb{R}$ or a subset of $\mathbb{R}$ and any point $a\in \mathbb{R}$, $p(a+)$ (resp. $p(a-)$) stands for the right (resp. left) limitation of $p$ at the point $a$ if it exists. Particularly, $p(\infty)$ or $p(-\infty)$ represents the left of right limitation of $p$ at $\infty$ or $-\infty$ if it exists. 

\section{Main results and their applications}\label{SEC2}

Throughout this paper, the state space $E$ is to be $\mathbb{R}^d$ or a domain of $\mathbb{R}^d$ for arbitrary natural number $d$, $m$ is a Radon measure fully supported on $E$. Furthermore, $(\EE,\FF)$ is a regular and strongly local Dirichlet form on $L^2(E,m)$ whose associated diffusion process is $\mathbf{X}=((X_t)_{t\geq 0}, (\mathbf{P}^x)_{x\in E})$, and $\FF_\mathrm{loc}$ is denoted as the local Dirichlet space of $\FF$ (Cf. \cite[\S3.2]{FOT11}). We always make the following basic assumption:
\begin{equation}\label{EQ2CCE}
	C_c^\infty(E)\subset \FF. 
\end{equation}
All the terminologies and notations of Dirichlet forms are referred to \cite{CF12, FOT11}. 

\subsection{One-dimensional cases}\label{SEC21}

We first consider $E=I=(a,b)$, an open interval of $\mathbb{R}$. Then $\mathbf{X}$ is to be a minimal diffusion process on $I$ with scaling function $\mathtt{s}$, speed measure $m$ and no killing inside (Cf.  \cite[Example~3.5.7]{CF12}). More precisely,
\begin{equation}\label{EQ2FFS}
\begin{aligned}
	\FF&=\FF^{(\ss,m)}_0, \\
	\EE(u,v)&=\EE^{(\ss,m)}(u,v):=\frac{1}{2}\int_I \frac{du}{d\ss}\frac{dv}{d\ss}d\ss,\quad u,v\in \FF, 
\end{aligned}
\end{equation}
where $\FF^{(\ss,m)}_0$ is the closure of $C_c^\infty\circ \ss:=\{\varphi\circ \ss: \varphi\in C_c^\infty(J)\}$ relative to the norm $\|\cdot\|_{\EE_1}$ and $J:=\ss(I)$. We refer the detailed expression of $\FF^{(\ss,m)}_0$ to \cite{FHY10}. 
 The scaling function $\mathtt{s}$ is strictly increasing and continuous, in other words, the Stieltjes measure $d\mathtt{s}$ is uniquely determined by $\mathbf{X}$. It is unique up to a constant. To avoid the occurrence of equivalence class, we take a fixed point $e\in I$ (if $I=\mathbb{R}$, set $e=0$) and suppose every scaling function equals $0$ at $e$.  Note that $\ss$ is not necessarily absolutely continuous. We also refer these terminologies to \cite{RY99, RW00}.
 
 We first have an equivalent characterization based on scaling function $\mathtt{s}$ to \eqref{EQ2CCE}. Let $\mathtt{t}:=\mathtt{s}^{-1}$ be the inverse function of $\mathtt{s}$ on $J$. 
 
\begin{lemma}\label{LM21}
	Let $(\EE,\FF)$ be the Dirichlet form in \eqref{EQ2FFS}. Then $C_c^\infty(I)\subset \FF$ if and only if $\mathtt{t}$ is absolutely continuous and $\mathtt{t}'\in L^2_\mathrm{loc}(J)$. 
\end{lemma}

Under the assumption \eqref{EQ2CCE}, one may easily check that the form
\[
	\bar{\EE}(u,v):=\EE(u,v),\quad u,v\in C_c^\infty(I)
\]
with the domain $C_c^\infty(I)$ is closable on $L^2(I,m)$. Denote its smallest closed extension by $(\bar{\EE},\bar{\FF})$. Then $(\bar{\EE},\bar{\FF})$ is a regular and strongly local Dirichlet form on $L^2(I,m)$ by  \cite[Theorem~3.1.2]{FOT11}. Moreover, $\bar{\FF}$ is a regular Dirichlet subspace of $\FF$ in the sense that
\[
	\bar{\FF}\subset \FF,\quad \EE(u,v)=\bar{\EE}(u,v),\quad u,v\in \bar{\FF}. 
\]
We refer some related studies about regular Dirichlet subspaces to \cite{FFY05, FHY10, LY15, SL15}. Particularly, it follows from \cite[Proposition~2.1]{SL15} that $(\bar{\EE},\bar{\FF})$ corresponds to another minimal diffusion process denoted by $\bar{\mathbf{X}}$ on $I$ with the same speed measure $m$ and no killing inside. So the critical questions are what the scaling function of $\bar{\mathbf{X}}$ is and whether $\bar{\FF}=\FF$ (or $\bar{\mathbf{X}}=\mathbf{X}$).

Before presenting our result, we need to introduce some notations.  Let $\lambda_I$ be the Lebesgue measure on $I$. Then we may write the Lebesgue decomposition of $d\mathtt{s}$ with respect to $\lambda_I$ on $I$, that is, there exist a positive function $g\in L^1_\mathrm{loc}(I)$ and another positive Radon measure $\kappa$ on $I$ such that 
\begin{equation}\label{EQ2SGL}
	d\ss=g\cdot \lambda_I+\kappa,\quad \kappa\perp \lambda_I. 
\end{equation}
Denote the absolutely continuous part of $d\ss$ by $d\bar{\ss}$, or
\begin{equation}\label{EQ2SXE}
	\bar{\ss}(x):=\int_e^x g(y)\lambda_I(dy),\quad x\in I.  
\end{equation}
We need to point out $\bar{\ss}$ is also a scaling function, i.e. a strictly increasing and continuous function on $I$. The continuity of $\bar{\ss}$ is obvious. Thus we need only to prove $g>0$ a.e. Since $\kappa\perp \lambda_I$, we may take a set $H$ of $\lambda_I$-zero measure such that $\kappa(I\setminus H)=0$. Let
\[
	Z_g:=\left\{x\in I\setminus H: g(x)=0  \right\}. 
\]
If $\lambda_I(Z_g)>0$, then
\[
	d\ss(Z_g)=\int_{Z_g}g(x)\lambda_I(dx)=0, 
\]
which contradicts the fact $\lambda_I\ll d\ss$ (Cf. Lemma~\ref{LM21}). 

\begin{theorem}\label{THM21}
Let $(\EE,\FF)$ be the Dirichlet form in \eqref{EQ2FFS}.  Assume \eqref{EQ2CCE} is satisfied. Then the absolutely continuous part $\bar{\ss}$ of $\ss$ is the scaling function of regular Dirichlet form $(\bar{\EE},\bar{\FF})$. In particular, $C_c^\infty(I)$ is a special standard core of $(\EE,\FF)$ if and only if its scaling function $\ss$ is absolutely continuous. 
\end{theorem}

\begin{remark}
	In this theorem (and hereafter), we always assume that the diffusion process has no killing inside. But this assumption is not essential. In fact, no matter whether the killing measure of $\mathbf{X}$ is present or not, we may do the resurrected transform or killing transform (Cf.  \cite[Theorem~5.1.6 and 5.2.17]{CF12}) on $(\EE,\FF)$. Then the two different situations (with or without killing inside) exchange, whereas any special standard core remains. 
\end{remark}

Naturally, we have the following expression:
\[
\begin{aligned}
	\bar{\FF}&=\FF^{(\bar{\ss},m)}_0,\\
	\bar{\EE}(u,v)&=\frac{1}{2}\int_I \frac{du}{d\bar{\ss}}\frac{dv}{d\bar{\ss}}d\bar{\ss},\quad u,v\in \bar{\FF}. 
\end{aligned}
\]
Furthermore, $C_c^\infty(I)$ and $C_c^\infty\circ \bar{\ss}:=\{\varphi\circ \bar{\ss}: \varphi\in C_c^\infty(\bar{J})\}$ ($\bar{J}:=\bar{\ss}(I)$) are simultaneously the special standard cores of $(\bar{\EE},\bar{\FF})$. 
Intuitively, the ``smooth part'' $\bar{\FF}$ of $\FF$ inherits exactly the ``good" (say absolutely continuous) part of $\FF$'s scaling function.   Particularly, if $\ss$ is absolutely continuous, then it naturally follows $\ss=\bar{\ss}$, and $\mathbf{X}$ equals $\bar{\mathbf{X}}$.

\subsection{Multi-dimensional cases: Cartesian product and skew product}\label{SEC22}

By employing Theorem~\ref{THM21}, we may deal with the analogous problems about diffusion processes on some special multi-dimensional domains. 

\subsubsection{Cartesian product}\label{SEC221}
The first case is the Cartesian product of one-dimensional diffusions on a rectangle cube. Let $d\geq 2$ be an integer and $\{I_i: 1\leq i\leq d\}$ a sequence of open intervals. For each $i$, $\mathbf{X}^i$ is a minimal diffusion process  on $I_i$ with scaling function $\ss_i$, speed measure $m_i$ and no killing inside. Set further
\[
\begin{aligned}
	E&:=I_1\times \cdots \times I_d, \\
	m&:= m_1\times \cdots \times m_d. 
\end{aligned}
\]
Assume that $\mathbf{X}$ is a diffusion process on $E$ such that its $i$-th component is equivalent to $\mathbf{X}^i$ in distribution and all the components are independent. More precisely, for any $x=(x_1,\cdots, x_d)\in E$ and $1\leq i\leq d$, set $f^i(x):=x_i$. Then $\mathbf{X}$ satisfies $f^i(\mathbf{X})\overset{d}{=} \mathbf{X}^i$ and $\{f^i(\mathbf{X}): 1\leq i\leq d\}$ are independent. Note that the diffusion process $\mathbf{X}$ is called the \emph{Cartesian product} of $\{\mathbf{X}^1,\cdots, \mathbf{X}^d\}$ (Cf. \cite[\S15]{S88}). Particularly, the boundary of $E$ is the trap of $\mathbf{X}$.

Denote the associated Dirichlet form of $\mathbf{X}^i$ on $L^2(I_i, m_i)$ by $(\EE^i,\FF^i)$.  The Cartesian product of Dirichlet forms was studied by H. \^Okura in \cite{O89, O97}. It is well known now that $\mathbf{X}$ is $m$-symmetric and its associated Dirichlet form $(\EE,\FF)$ on $L^2(E,m)$ is regular with a core of tensor product of $\mathbf{X}^i$s'. In other words, if $\mathcal{C}_i$ is a core of $(\EE^i,\FF^i)$,  then 
\[
	\mathcal{C}:=\mathcal{C}_1\otimes \cdots \otimes  \mathcal{C}_d=\text{span}\{u(x)=u_1(x_1)\cdots u_d(x_d): u_i\in \mathcal{C}_i, 1\leq i\leq d\}
\]
is a core of $(\EE,\FF)$. We also refer the similar result to \cite[Proposition~3.1]{LY15}. 

%\begin{figure}
%\centering
%\includegraphics[scale=0.45]{dp.jpg}
%\caption{The Cartesian product $\mathbf{X}$ of $\{\mathbf{X}^1,\cdots,\mathbf{X}^d\}$ on $E$}\label{SP}
%\end{figure}

To state an analogous result of Theorem~\ref{THM21}, we need to introduce more notations. Let $\mathtt{t}_i:=\mathtt{s}_i^{-1}$ be the inverse function of $\mathtt{s}_i$ on $J_i:=\mathtt{s}_i(I_i)$. Further set 
\[
\mathbf{s}:=(\mathtt{s}_1,\cdots, \mathtt{s}_d),\quad \mathbf{t}:=(\mathtt{t}_1,\cdots, \mathtt{t}_d). 
\]
Then $\mathbf{s}$ (with the inverse $\mathbf{t}$) is a homeomorphism between $E$ and $U:=J_1\times \cdots \times J_d$. We say $\mathbf{s}$ (resp. $\mathbf{t}$) is absolutely continuous if each $\mathtt{s}_i$ (resp. $\mathbf{t}_i$) is absolutely continuous. If $\mathbf{t}$ is absolutely continuous, write
\[
	\mathbf{t}':=(\mathtt{t}'_1,\cdots, \mathtt{t}'_d). 
\]
Clearly, $\mathbf{t}'\in L^2_\mathrm{loc}(U)$ if and only if each $\mathtt{t}'_i\in L^2_\mathrm{loc}(J_i)$ for $1\leq i\leq d$. We then have an analogous characterization of Lemma~\ref{LM21} to \eqref{EQ2CCE}. 

\begin{lemma}\label{LM24}
	Let $(\EE,\FF)$ be the associated Dirichlet form of Cartesian product $\mathbf{X}$ of minimal diffusion processes $\{\mathbf{X}^1,\cdots,\mathbf{X}^d\}$. Then $C_c^\infty(E)\subset \FF$ if and only if $\mathbf{t}$ is absolutely continuous and $\mathbf{t}'\in L^2_\mathrm{loc}(U)$. 
\end{lemma}

Similarly, denote the absolutely continuous part of $\mathbf{s}$ by $\bar{\mathbf{s}}=(\bar{\mathtt{s}}_1,\cdots, \bar{\mathtt{s}}_d)$, i.e. each $\bar{\ss}_i$ is the absolutely continuous part of $\mathtt{s}_i$ (as \eqref{EQ2SXE}). 

\begin{theorem}\label{THM25}
Let $(\EE,\FF)$ be that in Lemma~\ref{LM24}. Assume \eqref{EQ2CCE} is satisfied. Then the smallest closed extension denoted by $\bar{\FF}$ of $C_c^\infty(E)$ in $\FF$ corresponds to the Cartesian product of $\{\bar{\mathbf{X}}^1,\cdots, \bar{\mathbf{X}}^d\}$, where $\bar{\mathbf{X}}^i$ is the minimal diffusion process with scaling function $\bar{\mathtt{s}}_i$ and speed measure $m_i$. Particularly, $C_c^\infty(E)$ is a special standard core of $(\EE,\FF)$ if and only if $\mathbf{s}$ is absolutely continuous. 
\end{theorem}

Roughly speaking, the smallest closed extension $\bar{\FF}$ of $C_c^\infty(E)$ in $\FF$ inherits the absolutely continuous part $\bar{\mathbf{s}}$ of $\mathbf{s}$, which is similar to the one-dimensional case of Theorem~\ref{THM21}. 

\subsubsection{Skew product}\label{SEC222}

The skew product of two Markov processes was first raised by Galmarino in \cite{G62} when he investigated the isotropic diffusion process on $\mathbb{R}^3$. His main result indicates that every isotropic diffusion process may be written as the form of
\[
	(r_t, \vartheta_{A_t})_{t\geq 0},
\]
where $r=(r_t)_{t\geq 0}$ is the radius part of the process, and $\vartheta$ is  an independent (of $r$) spherical Brownian motion that runs with a clock $(A_t)_{t\geq 0}$ depending on the radial path. Mathematically, $A=(A_t)_{t\geq 0}$ is a positive continuous additive functional (PCAF in abbreviation) of $r$. It\^o and McKean also introduced this idea in their masterpiece \cite{IM74}. Then Fukushima and Oshima in \cite{FO89} constructed the associated Dirichlet form of skew product of two symmetric diffusion processes, and \^Okura in \cite{O89, O97} extended their representation to general situations. Based on the results of \^Okura, one of the authors in \cite{LY15-2} studied the regular Dirichlet subspaces of skew product of two Markov processes. 

Now we consider the state space $E=I\times S^{d-1}$,  where $I=(a,b)$ is an open subinterval of $(0,\infty)$ and $S^{d-1}$ is the surface of $d$-dimensional unit ball with $d\geq 2$. Clearly $E$ is a domain in $\mathbb{R}^d$. Let $\mathbf{X}=(X_t)_{t\geq 0}$ be a diffusion process on $E$ which enjoys the form
\[
	X_t=(r_t, \vartheta_{A_t}),\quad t\geq 0,
\]
where $r=(r_t)_{t\geq 0}$ is a minimal diffusion process on $I$, $\vartheta$ is  a spherical Brownian motion on $S^{d-1}$ that independent of $r$, and $A=(A_t)_{t\geq 0}$ is a PCAF of $r$ whose Revuz measure, denoted by $\mu_A$, is Radon on $I$. Note that the boundary $\{a,b\}\times S^{d-1}$ of $E$ is the ``trap" of $\mathbf{X}$. Naturally, $\mathbf{X}$ is isotropic in the sense that if $\varphi$ is a rotation on $E$ (i.e. the polar coordinate representation of $\varphi$ is $(\rho, {\theta})\mapsto ( \rho, {\theta}+{\alpha})$ for a fixed ${\alpha}$), then $\varphi(\mathbf{X})=\mathbf{X}$ in distribution.  Particularly, if $I=(0,\infty)$, then $\mathbf{X}$ is exactly the isotropic diffusion process in \cite{G62} (with a little more restriction on the ``clock"). 

%\begin{figure}
%\centering 
%\includegraphics[scale=0.5]{Shell.jpg}
%\caption{Skew product}\label{FIG2}
%\end{figure}

Let $\mathtt{s}$ and $m_0$ be the scaling function and speed measure of $r$. Denote the inverse function of $\mathtt{s}$ by $\mathtt{t}$. 
 Further let $\sigma$ be the uniform distribution on $S^{d-1}$. Then $r$ is $m_0$-symmetric and $\vartheta$ is $\sigma$-symmetric, thus $\mathbf{X}$ is $m$-symmetric (Cf. \cite[Proposition~3.1]{FO89}), where $m(dx)=m_0(d\rho)\sigma(d\omega)$ with $x=(\rho,\omega)$ ($\rho\in I, \omega\in S^{d-1})$. Furthermore, the associated Dirichlet form $(\EE,\FF)$ of $\mathbf{X}$ is regular on $L^2(E,m)$ and admits the following expression:  for any $u\in \FF\cap C_c(E)$, 
\begin{equation}\label{EQ22OMU}
\begin{aligned}
	&[\omega \mapsto u(\cdot, \omega)\in \FF^1] \in L^2(S^{d-1},\sigma;\FF^1),\\
	&[\rho \mapsto u(\rho,\cdot)\in \FF^2] \in L^2(I,m_0;\FF^2)
\end{aligned}	
\end{equation}
and 
\[
	\EE(u,u)=\int_{S^{d-1}}\EE^1(u(\cdot, \omega),u(\cdot,\omega))\sigma(d\omega)+\int_{I}\EE^2(u(\rho,\cdot),u(\rho,\cdot))\mu_A(d\rho),
\]
where $(\EE^1,\FF^1)$ is the associated Dirichlet form of $r$ on $L^2(I,m_0)$, and $(\EE^2,\FF^2)$ is that of $\vartheta$ on $L^2(S^{d-1},\sigma)$. Note that $L^2(M,\xi;\mathbf{H})$ means the real $L^2$-space of $\mathbf{H}$-valued functions, where $M=S^{d-1}$ or $I$, $\xi=\sigma$ or $m_0$, and $\mathbf{H}=\FF^i$ with the inner product $\EE^i_1(\cdot,\cdot)$ for $i=1$ or $2$.
We refer more details about $(\EE,\FF)$ to \cite{FO89, O89, O97} and also \cite{LY15-2}. 

\begin{lemma}\label{LM26}
Let $(\EE,\FF)$ be the associated Dirichlet form on $L^2(E,m)$ of skew product diffusion process $\mathbf{X}$ given above.  Then $C_c^\infty(E)\subset \FF$ if and only if $\mathtt{t}$ is absolutely continuous and $\mathtt{t}'\in L^2_\mathrm{loc}(J)$, where $J:=\mathtt{s}(I)$. 
\end{lemma}
 
Similarly, let $\bar{\mathtt{s}}$ be the absolutely continuous part of $(r_t)_{t\geq 0}$'s scaling function $\mathtt{s}$, and $\bar{r}=(\bar{r}_t)_{t\geq 0}$ the minimal diffusion process on $I$ with scaling function $\bar{\mathtt{s}}$ and speed measure $m_0$. Intuitively, $\bar{r}$ is the absolutely continuous part of $r$. Next, we may state another analogous result of Theorem~\ref{THM21} about the skew product outlined above.

\begin{theorem}\label{THM27}
Let $(\EE,\FF)$ be the Dirichlet form in Lemma~\ref{LM26} and assume $C_c^\infty(E)\subset \FF$. Then the smallest closed extension, denoted by $\bar{\FF}$, of $C_c^\infty(E)$ in $\FF$ is the associated Dirichlet space of skew product diffusion process
\begin{equation}\label{EQ22XTR}
	\bar{X}_t=(\bar{r}_t,\vartheta_{\bar{A}_t}),\quad t\geq 0,
\end{equation}
where $\bar{A}$ is a PCAF of $\bar{r}$ such that its Revuz measure $\mu_{\bar{A}}$ equals $\mu_A$, i.e. $\mu_{\bar{A}}=\mu_A$. Particularly, $C_c^\infty(E)$ is a special standard core of $(\EE,\FF)$ if and only if $\mathtt{s}$ is absolutely continuous. 
\end{theorem}

\subsection{Multi-dimensional cases: energy forms}\label{SEC23}

One may expect to extend the characterizations in \S\ref{SEC21} and \S\ref{SEC22} to general diffusion processes on an arbitrary domain of $\mathbb{R}^d$. But this is difficult because we do not have a complete description of multi-dimensional diffusion process like the scaling function and speed measure of one-dimensional case.  The Cartesian product and skew product are two exceptions, since they are indeed achieved by transforms of one-dimensional diffusions. Nevertheless, if we force the diffusion process to have a Brownian motion as its martingale part in the Fukushima's decomposition \eqref{EQ1XTX} besides the condition \eqref{EQ2CCE}, then we may also tell some interesting stories. 

Assume $E=U$, a domain of $\mathbb{R}^d$ with $d\geq 2$, and $m$ is a Radon measure fully supported on $U$. The process $\mathbf{X}=(X_t)_{t\geq 0}$ is an absorbing $m$-symmetric diffusion process without killing inside on $U$, whose associated Dirichlet form $(\EE,\FF)$ is regular (and obviously strongly local) on $L^2(U,m)$. The word ``absorbing" means $U^c$ is the trap of $\mathbf{X}$. Further assume our elementary assumption:
\[
	C_c^\infty(U)\subset \FF. 
\]
Since this condition implies each $i$th coordinate function $f^i\in \FF_\mathrm{loc}$ (i.e. $f^i(x)=x_i$ for $x=(x_1,\cdots, x_d)\in U$) with $1\leq i\leq d$, it follows that $\mathbf{X}$ enjoys the Fukushima's decomposition with resepct to $f:=(f^1,\cdots, f^d)$: 
\begin{equation}\label{EQ23FXT}
f(X_t)-f(X_0)=M^{[f]}_t+N^{[f]}_t,\quad t\geq 0,
\end{equation}
where $M^{[f]}=\left(M^{[f^1]}, \cdots, M^{[f^d]} \right)$ and $N^{[f]}=\left(N^{[f^1]}, \cdots, N^{[f^d]} \right)$ are the martingale part locally of finite energy and locally zero-energy part respectively. We say $M^{[f]}$ is equivalent to a Brownian motion if for q.e. $x\in U$, $M^{[f]}$ is equivalent to a standard Brownian motion (the sample path space is endowed with the measure $\mathbf{P}^x$) before $\zeta$, where $\zeta$ is the life time of $\mathbf{X}$. 

\begin{lemma}\label{LM28}
Let $(\EE,\FF)$ be the regular and strongly local Dirichlet form on $L^2(U,m)$, whose associated diffusion process is $\mathbf{X}$. Assume $C_c^\infty(U)\subset \FF$. Then the martingale part $M^{[f]}$ in \eqref{EQ23FXT} is a Brownian motion, if and only if for any $u,v\in C_c^\infty(U)$, 
\begin{equation}\label{EQ23EUV}
	\EE(u,v)=\frac{1}{2}\int_U \nabla u(x)\cdot \nabla v(x) m(dx).
\end{equation}
\end{lemma}

Note that Fukushima in \cite{Fu82} (the corollary in its appendix) already proved a special case, i.e. $U=\mathbb{R}^d$. The proof of Lemma~\ref{LM28} is essentially the same. We also point out the relevant discussion may retrospect to Beurling and Deny's second formula in \cite{BD59}. 
On the other hand, the form \eqref{EQ23EUV} is usually named the energy form (such as \cite{Fu85}). If additionally $C_c^\infty(U)$ is a special standard core of $(\EE,\FF)$, then $\mathbf{X}$ is also called an (absorbing) distorted Brownian motion on $U$.

It is well known that $N^{[f]}$ is not necessarily of bounded variation. We refer some detailed studies about additive functionals of bounded variation to \cite[\S5.4]{FOT11}. Note that if the zero-energy part in Fukushima's decomposition is of bounded variation, then Fukushima's decomposition  coincides with the semi-martingale decomposition. Intuitively, that $N^{[f]}$ is of bounded variation indicates a ``good" regularity of the sample path of non-martingale part of $\mathbf{X}$. Furthermore, when $N^{[f^i]}$ is of bounded variation, it may be written as the difference of two PCAFs. Let $\mu_1$ and $\mu_2$ be the associated Revuz measures of these two PCAFs. Then 
\[
	\mu_{N^{[f^i]}}:=\mu_1-\mu_2
\]
is called the smooth signed measure of $N^{[f^i]}$. 
The following theorem states the equivalence between this fact and the denseness of $C_c^\infty(U)$ in $\FF$. 

\begin{theorem}\label{THM29}
Let $(\EE,\FF)$ be in Lemma~\ref{LM28}. Assume $C_c^\infty(U)\subset \FF$ and $M^{[f]}$ in \eqref{EQ23FXT} is equivalent to a Brownian motion. Assume further $m(dx)=\rho(x)dx$ with a strictly positive $C^\infty$-density function $\rho$. Then the following assertions are equivalent:
\begin{description}
\item[(1)] $C_c^\infty(U)$ is a special standard core of $(\EE,\FF)$;
\item[(2)] $N^{[f]}$ in \eqref{EQ23FXT} is of bounded variation, and each $\mu_{N^{[f^i]}}$ is Radon; 
\item[(3)] $C_c^\infty(U)\subset \mathcal{D}(\mathcal{L})$, where $\mathcal{L}$ is the associated self-adjoint operator of $(\EE,\FF)$ and $\mathcal{D}(\mathcal{L})$ is its domain.
\end{description}
\end{theorem}

\begin{remark}
	Under all the conditions of Theorem~\ref{THM29} (even if $\rho\equiv 1$), we may still find some examples, in which the equivalent assertions above are not satisfied (i.e. $C_c^\infty(U)$ is not a special standard core of $(\EE,\FF)$). One simple way is via the Cartesian product illustrated in \S\ref{SEC21}, and we refer it to Example~\ref{EXA55}.
\end{remark}
	
One may doubt the condition on the symmetric measure $m$ is too rigorous, while another proposition below implies the necessity of the absolute continuity of $m$. 

\begin{proposition}\label{PRO211}
Assume $C_c^\infty(U)\subset \FF$ and $M^{[f]}$ in \eqref{EQ23FXT} is equivalent to a Brownian motion. If $N^{[f]}$ is of bounded variation and each $\mu_{N^{[f^i]}}$ is Radon, then $m$ is absolutely continuous with respect to the Lebesgue measure $\lambda_U$ on $U$. 
\end{proposition}

Furthermore, if the density function $\rho$ of $m$ does not satisfy the assumption in Theorem~\ref{THM29} (even if $\rho$ is smooth almost everywhere except only one point), we may find some counterexample to break the equivalence between \textbf{(1)} and \textbf{(2)}, see Example~\ref{EXA56}.

\subsection{Applications}\label{SEC24}

In this section, we shall introduce some applications of the results above in the theory of Markov processes and stochastic analysis. 

\subsubsection{Characterization of isotropic diffusion processes}\label{SEC241}

The special case, i.e. that $\mathtt{s}$ is absolutely continuous, of  Theorem~\ref{THM21} provides two different special standard core (say $C_c^\infty(I)$ and $C_c^\infty\circ \mathtt{s}$) for the Dirichlet form $(\EE,\FF)$. Thus they offer two useful viewpoints to observe the same diffusion process. By employing this fact, one of the authors in \cite{LY15-2}  characterized the isotropic diffusion processes on $\mathbb{R}^d$ and their associated Dirichlet forms via a very short proof, see \cite[Lemma~4.3]{LY15-2}. In the mean time, the proof of the uniqueness of regular extension raised in \cite[Theorem~4.4]{LY15-2} also relies on the transforms between these two different expressions. We refer more details to relevant studies of \cite{LY15-2}.

\subsubsection{Uncountable solutions of SDEs}\label{SEC242}

In this subsection, let us consider the stochastic differential equation (SDE in abbreviation)
\begin{equation}\label{EQ24DXT}
	dX_t=b(X_t)dt+\sigma(X_t)dB_t
\end{equation}
on $\mathbb{R}^d$ with an integer $d\geq 1$ and $B=(B_t)_{t\geq 0}$ a standard Brownian motion on $\mathbb{R}^d$. It is well known if $\sigma\sigma^T$ is uniformly positive definite, bounded and continuous, and $b$ is bounded Borel measurable, then solution of \eqref{EQ24DXT} exists and is unique (see \cite{IW81}). However, there are very few examples in which the uniqueness breaks. As we know, It\^o and Watanabe in \cite{IW78} raised an example, i.e. $b(x)=3x^{1/3}, \sigma(x)=3x^{2/3}$, and proved it has uncountable strong solutions. But these solutions are all starting from $0$.

Now we shall give a class of coefficients $\{b, \sigma\}$ such that \eqref{EQ24DXT}  has infinite (explicitly, uncountable) different solutions that may start from any place.  
Without loss of generality, we may only consider the case $d=1$, i.e. the state space is $\mathbb{R}$. (In fact, if we replace $\mathbb{R}$ by an open interval $I$, then all the conclusions below are still right.) For multi-dimensional cases, the Cartesian product method is one easy way to give such a vector $b$ and matrix $\sigma$. 

Let $\mathbf{X}$ be an irreducible diffusion process on $\mathbb{R}$ without killing inside. Then $\mathbf{X}$ can be characterized by a scaling function $\mathtt{s}$ and speed measure $m$, and its associated Dirichlet form $(\EE,\FF)$ on $L^2(\mathbb{R},m)$ may be written as \eqref{EQ2FFS} via replacing $I$ by $\mathbb{R}$.  Assume
\begin{description}
\item[(H1)] $C_c^\infty(\mathbb{R})\subset \FF$;
\item[(H2)] $\mathtt{s}$ is not absolutely continuous. 
\end{description}
Let $\mathtt{t}$, $(\bar{\EE},\bar{\FF})$ and $\bar{\mathbf{X}}$ be the notations in \S\ref{SEC21}. Note that \textbf{(H1)} means $\mathtt{t}$ is absolutely continuous and $\mathtt{t}'\in L^2_\mathrm{loc}$, and \textbf{(H2)} indicates $\bar{\FF}\neq \FF$ (or $\bar{\mathbf{X}}\neq \mathbf{X}$ in distribution). Set
\begin{equation}\label{EQ24MXT}
	\tilde{m}(dx):=\mathtt{t}'\circ \mathtt{s}(x)dx,
\end{equation}
which is a positive Radon measure on $\mathbb{R}$ (note that $\tilde{m}(dx)=\left(\mathtt{t}'\circ \mathtt{s}\right)^2(x)d\mathtt{s}(x)$ and $\mathtt{t}'\in L^2_\mathrm{loc}$). Assume further
\begin{description}
\item[(H3)] $\tilde{m}\ll m$;
\item[(H4)] an a.e. version of $\mathtt{t}'$ is locally of bounded variation on $J$, and $d\mathtt{t}_*(\mathtt{t}')\ll m$, where $d\mathtt{t}_*(\mathtt{t}')$ is the image measure of $d\mathtt{t}'$ via the transform $\mathtt{t}: J\rightarrow \mathbb{R}$.
\end{description}
Under \textbf{(H3)}, denote
\begin{equation}\label{EQ24BDM}
	\sigma:= \left(\frac{d\tilde{m}}{dm}\right)^{1/2}.
\end{equation}
Clearly, $\sigma\in L^2_\mathrm{loc}(\mathbb{R},m)$. Under \textbf{(H4)}, denote
\begin{equation}\label{EQ24BDT}
	b:=\frac{1}{2}\cdot\frac{d\mathtt{t}_*(\mathtt{t}')}{dm}.
\end{equation}

Before presenting our main result, we need to make some notations. Recall that a regular Dirichlet form $(\EE',\FF')$ on $L^2(\mathbb{R},m)$ is said to be a regular Dirichlet subspace of $(\EE,\FF)$ provided $\FF'\subset \FF$ and $\EE(u,v)=\EE'(u,v)$ for any $u,v\in \FF'$. If so, we write $(\EE',\FF')\prec (\EE,\FF)$. Under \textbf{(H1)}, denote a set of regular Dirichlet forms on $L^2(\mathbb{R},m)$ by
\[
	\mathfrak{D}:=\left\{(\EE',\FF'): (\EE',\FF')\prec (\EE,\FF)\text{ and }C_c^\infty(\mathbb{R})\subset \FF'  \right\}.
\]
Clearly $(\bar{\EE},\bar{\FF})\in \mathfrak{D}$ and it is exactly the smallest one in it (i.e. if $(\EE',\FF')\in \mathfrak{D}$, then $(\bar{\EE},\bar{\FF})\prec (\EE',\FF')$). Moreover, under \textbf{(H2)}, the elements of $\mathfrak{D}$ are not unique (say $(\EE,\FF), (\bar{\EE},\bar{\FF})\in \mathfrak{D}$ but $\FF\neq \bar{\FF}$). Indeed, $\mathfrak{D}$ has uncountable elements.

\begin{lemma}\label{LM212}
	Under \textbf{(H1)} and \textbf{(H2)}, $\mathfrak{D}$ has uncountable elements. 
\end{lemma}

Now we may state our main result as follows. A diffusion process $\mathbf{Y}=(Y_t, \mathbf{Q}^x)_{t\geq 0, x\in \mathbb{R}}$ is called a solution of \eqref{EQ24DXT}, if for any initial distribution $\xi$ on $\mathbb{R}$, there exist another probability measure space $(\tilde{\Omega}, \tilde{\mathcal{G}}, \tilde{\mathbf{Q}}^\xi)$ with a filtration $(\tilde{\mathcal{G}}_t)_{t\geq 0}$, a $\tilde{\mathcal{G}}_t$-adapted continuous process $(\tilde{Y}_t)_{t\geq 0}$ and a $\tilde{\mathcal{G}}_t$-adapted standard Brownian motion $(\tilde{B}_t)_{t\geq 0}$ such that
\begin{description}
\item[(1)] $(\tilde{Y}_t)_{t\geq 0}$ is equivalent to $(Y_t)_{t\geq 0}$ that endowed with the probability measure $\mathbf{Q}^\xi$, where  	
\[
\mathbf{Q}^\xi(\cdot):=\int\xi(dx)\mathbf{Q}^x(\cdot);
\]
\item[(2)] $\tilde{\mathbf{Q}}^\xi$-a.s., for any $t\geq 0$, 
	\[
	 \tilde{Y}_t-\tilde{Y}_0=\int_0^t b(\tilde{Y}_s)ds+\int_0^t \sigma(\tilde{Y}_s)d\tilde{B}_s.
	\]
\end{description}  
Note that the first condition implies $\tilde{Y}_0\overset{d}{=}\xi$. Particularly, when $\sigma=1_\mathbb{R}$ (or $\sigma=1$ if no confusion causes), we write \eqref{EQ24DXT} as
\[
	dX_t=b(X_t)dt+dB_t
\]
for convinience, where $dB_t$ actually means $1_{\mathbb{R}}(X_t)dB_t$.

\begin{theorem}\label{THM214}
Let $\mathbf{X}$ be the diffusion process on $\mathbb{R}$ associated with the Dirichlet form $(\EE,\FF)$ on $L^2(\mathbb{R},m)$. Assume \textbf{(H1)}, \textbf{(H2)}, \textbf{(H3)}, \textbf{(H4)} are satisfied. Let $b,\sigma$ be given in \eqref{EQ24BDM} and \eqref{EQ24BDT}. Then for any $(\EE',\FF')\in \mathfrak{D}$, its associated diffusion process $\mathbf{X}'$ is a solution of the following SDE:
\begin{equation}\label{EQ24DXB}
	dX_t=b(X_t)dt+\sigma(X_t)dB_t,
\end{equation}
where $B=(B_t)_{t\geq 0}$ is a standard Brownian motion. 
\end{theorem}

The assumptions \textbf{(H3)} and \textbf{(H4)} seem to be somewhat abstract, but they may be replaced by another two accessible assumptions:
\begin{description}
\item[(H3')] $m(dx)=h(x)dx$ for some a.e. strictly positive function $h\in L^1_\mathrm{loc}(\mathbb{R})$;
\item[(H4')] $\mathtt{t}\in C^1(J)$ and $\mathtt{t}'$ is absolutely continuous. 
\end{description}
Note that  \textbf{(H3')} is stronger than \textbf{(H3)}, and \textbf{(H4')} is stronger than \textbf{(H1)}. Actually we have the following lemma. 

\begin{lemma}\label{LM214}
	$\textbf{(H3')}+\textbf{(H4')}\Rightarrow \textbf{(H3)}+\textbf{(H4)}$.
\end{lemma}

In the end of this section, let us compute more concrete expressions of $b, \sigma$ under the assumptions \textbf{(H2)}, \textbf{(H3')} and \textbf{(H4')}. Clearly,
\[
	\sigma=\left(\frac{\mathtt{t}'\circ \mathtt{s}}{h}\right)^{1/2}.
\]
It follows from \textbf{(H4')} that $d\mathtt{t}'=\mathtt{t}''(x)dx$ and $d\mathtt{t}_*(\mathtt{t}')=\left(\mathtt{t}''(x)dx\right)\circ \mathtt{s}$. We assert $d\mathtt{t}_*(\mathtt{t}')\ll \tilde{m}$. In fact, note that $\tilde{m}(dx)=\left(\mathtt{t}'(y)^2dy\right)\circ \mathtt{s}(x)$, which implies $d\mathtt{t}_*(\mathtt{t}')\ll \tilde{m}$ is equivalent to $\mathtt{t}''(x)dx\ll \mathtt{t}'(x)^2dx$. Let $A$ be a set such that 
\[
	\int_A \mathtt{t}'(x)^2dx=0.
\]
Since $\mathtt{t}'\geq 0$, it follows that $A\subset Z_{\mathtt{t'}}:=\{x: \mathtt{t}'(x)=0\}$. For a.e. $z\in A$, $\mathtt{t}'$ is differential at $z$, whereas $\mathtt{t}'\geq 0$ and $\mathtt{t}'(z)=0$. Thus we may deduce that $\mathtt{t}''(z)=0$. That indicates
\[
	\int_A \mathtt{t}''(x)dx=0.
\]
Furthermore, 
\[
	2b=\frac{d\mathtt{t}_*(\mathtt{t}')}{dm}=\frac{d\mathtt{t}_*(\mathtt{t}')}{d\tilde{m}}\cdot \frac{d\tilde{m}}{dm}=\frac{\mathtt{t}''}{(\mathtt{t}')^2}\circ \mathtt{s} \cdot \frac{\mathtt{t}'\circ \mathtt{s}}{h}=\frac{\mathtt{t}''\circ \mathtt{s}}{(\mathtt{t}'\circ \mathtt{s})\cdot h}.
\]
Particularly, if $m=\tilde{m}$ (i.e. $\EE$ is an energy form), then $\sigma\equiv 1$ and $2b=\left(\mathtt{t}''/(\mathtt{t}')^2\right)\circ \mathtt{s}$. 
Surely, we need to point out $\mathtt{t}'\circ \mathtt{s}>0$ a.e. (so \textbf{(H3)} is satisfied). In fact, from $\mathtt{t}'\circ \mathtt{s}\geq 0$, and ($|\cdot|$ denotes the Lebesgue measure)
\[
|\{x: \mathtt{t}'\circ \mathtt{s}(x)=0\}|=\int_{\mathtt{t}'\circ \mathtt{s}=0}dx= \int_{\mathtt{t}'\circ \mathtt{s}=0} \mathtt{t}'\circ \mathtt{s}d\mathtt{s}=0,
\]
we may deduce it easily. Therefore, we have the following corollary.

\begin{corollary}\label{COR215}
Let $\mathbf{X}$ and $(\EE,\FF)$ be in Theorem~\ref{THM214}. Assume \textbf{(H2)}, \textbf{(H4')} are satisfied, and $m=\tilde{m}$. Then for any $(\EE',\FF')\in \mathfrak{D}$, its associated diffusion process $\mathbf{X}'$ is a solution of the following SDE:
\begin{equation}\label{EQ24XTT}
dX_t=\frac{1}{2}\cdot \left(\frac{\mathtt{t}''}{(\mathtt{t}')^2}\circ \mathtt{s}\right)(X_t)dt+dB_t,
\end{equation}
where $B=(B_t)_{t\geq 0}$ is a standard Brownian motion. 
\end{corollary}

One may worry that the assumptions above are so strict ($\ss,\tt$ must also be continuous and strictly increasing) that no examples would obey them. Hence we shall give an example that satisfies \textbf{(H2)}, \textbf{(H3')} and \textbf{(H4')}  in Example~\ref{EXA52} by using a (in fact, any) compact generalized Cantor set.

\section{Proofs of \S\ref{SEC21} and \S\ref{SEC22}}\label{SEC3}

In this section, we shall prove the denseness of $C_c^\infty(E)$ in the Dirichlet spaces outlined in \S\ref{SEC21} and \S\ref{SEC22}. These Dirichlet spaces share the common feather, i.e. they are induced by one or several scaling functions. 

\subsection{One-dimensional cases}

We first consider the one-dimensional cases. Lemma~\ref{LM21} characterized the basic assumption \eqref{EQ2CCE} via the inverse $\mathtt{t}$ of scaling function $\mathtt{s}$. As we noted before, \eqref{EQ2CCE} is also equal to \eqref{EQ1FFL}. We shall write down these three conditions for handy reference and prove their equivalence:
\begin{description}
\item[(f)] $f\in \FF_\mathrm{loc}$;
\item[(t)] $\mathtt{t}$ is absolutely continuous, and $\mathtt{t}'\in L^2_\mathrm{loc}(J)$;
\item[(c)] $C_c^\infty(I)\subset \FF$. 
\end{description}

\begin{proof}[Proof of Lemma~\ref{LM21}]
In fact, we shall prove $\textbf{(f)}\Leftrightarrow \textbf{(t)}\Leftrightarrow \textbf{(c)}$. Firstly, assume \textbf{(f)} holds. Let $I'$ be an arbitrary relatively compact open subset of $I$, then there exists a function $f^{I'}\in \mathcal{F}$ such that $f^{I'}(x)=x$ for any $x\in I'$ and thus
	\[
		f^{I'}(x) = (f^{I'}\circ \mathtt{t}) (\mathtt{s}(x))=\mathtt{t}(\mathtt{s}(x)), \quad x\in I'. 
	\]
Since $f^{I'}$ is absolutely continuous with respect to $\mathtt{s}$ on $I'$ (Cf. \cite[\S2]{FHY10} or \cite[(2.2.28)]{CF12}), it follows that $\mathtt{t}$ is absolutely continuous on $J':=\mathtt{s}(I')$ and 
	\[
		\int_{J'} \left(\mathtt{t}'(x)\right)^2 dx=\int_{I'} \left(\frac{df^{I'}}{d\mathtt{s}}\right)^2 d\mathtt{s} \leq  \int_I \left(\frac{df^{I'}}{d\mathtt{s}}\right)^2 d\mathtt{s}<\infty.
	\]
Since $I'$ is arbitrary, we may conclude $\mathtt{t}$ is absolutely continuous on $J$ and $\mathtt{t}'\in L^2_{\text{loc}}(J)$. That implies \textbf{(t)}. 

Next, assume \textbf{(t)} holds. For any $\phi\in C_c^\infty(I)$, it follows from 
\[
	\phi=\phi\circ \mathtt{t}\circ \mathtt{s}
\]
and $\mathtt{t}$ is absolutely continuous that $\phi$ is absolutely continuous with respect to $\mathtt{s}$. Moreover,
\[\begin{aligned}
	\int_I \left(\frac{d\phi}{d\mathtt{s}}\right)^2d\mathtt{s}=\int_J \left(\frac{d\phi\circ \mathtt{t}}{dx}\right)^2 dx=\int_J {\phi'(\mathtt{t}(x))}^2{\mathtt{t}'(x)}^2dx.
\end{aligned}\]
Since $\phi'$ is bounded with compact support and $\mathtt{t}'\in L^2_{\text{loc}}(J)$, we have 
\[\int_I \left(\frac{d\phi}{d\mathtt{s}}\right)^2d\mathtt{s}<\infty.\]
 Hence $\phi \in \mathcal{F}$, and that implies \textbf{(c)}.

Finally, assume \textbf{(c)} holds. For any relatively compact open subset $I'$ of $I$, by using the Urysohn's lemma, we may take a function $h\in C_c^\infty(I)$ with $h|_{
\bar{I}'}\equiv 1$. Define
\begin{equation}\label{FIF}
	f^{I'}:= f\cdot h.
\end{equation}
Clearly $f^{I'}\in C_c^\infty(I)\subset\mathcal{F}$ and 
\[
	f^{I'}(x)=x, \quad x\in I',
\]
which indicates $f\in \mathcal{F}_{\text{loc}}$. That completes the proof. 
\end{proof}

\begin{remark}\label{RM31}
Under the condition \textbf{(c)} (equivalently \textbf{(f)} or \textbf{(t)}), we may obtain another expression of $\EE(u,v)$ for any $u,v\in C_c^\infty(I)$. Indeed, for any $u,v\in C_c^\infty(I)$, it follows from \eqref{EQ2FFS} that
\[
	\EE(u,v)=\frac{1}{2}\int_I \left(\frac{du}{dx}\cdot \frac{dx}{d\ss}\right) \left(\frac{dv}{dx}\cdot\frac{dx}{d\ss}\right)d\ss=\frac{1}{2}\int_I u'(x)v'(x)(\mathtt{t}'\circ \ss(x))^2d\ss(x).
\]
Note that $dx$ denotes the Lebesgue measure and $dx=d\tt\circ \ss(x)=(\tt'\circ \ss)d\ss$. Recall that in \eqref{EQ24MXT} we use $\tilde{m}$ to stand for the measure $\tt'\circ\ss(x)dx= (\tt'\circ \ss)^2d\ss$, thus we have for any $u,v\in C_c^\infty(I)$, 
\[
	\EE(u,v)=\frac{1}{2}\int_I u'(x)v'(x)\tilde{m}(dx). 
\]
Particularly, if $m=\tilde{m}$ (say Corollary~\ref{COR215}),  then the Dirichlet form $(\bar{\EE},\bar{\FF})$ (i.e. the smallest closed extension of $C_c^\infty(I)$ in $(\EE,\FF)$) is an energy form on $L^2(I,m)$. 
\end{remark}

Next, we take to prove Theorem~\ref{THM21}. Before starting the detailed proof, we need to explain some facts about the regular Dirichlet subspaces of $(\EE,\FF)$. Any regular Dirichlet subspace $(\EE',\FF')$ of $(\EE,\FF)$ still corresponds to a minimal diffusion process with the same speed measure $m$ and another scaling function $\hat{\ss}$ such that
\begin{equation}\label{EQ31DSS}
	\frac{d\hat{\ss}}{d\ss}=0\text{ or }1,\quad d\ss\text{-a.e.},
\end{equation}
and vice versa.
In other words, $(\EE',\FF')=\left(\EE^{(\hat{\ss},m)}, \FF^{(\hat{\ss},m)}_0\right)$ (see \eqref{EQ2FFS}), and all regular Dirichlet subspaces of $(\EE,\FF)$ has a one-to-one correspondence with the scaling functions satisfying \eqref{EQ31DSS}. This characterization can be found in \cite[Theorem~4.1]{FHY10} and  \cite[Proposition~2.2]{SL15}. Note that the scaling function $\bar{\ss}$ defined by \eqref{EQ2SXE} satisfies
\[
	\frac{d\bar{\ss}}{d\ss}=1_{I\setminus H},
\]
where $H$ is the support of $\kappa$ in \eqref{EQ2SGL}. Particularly, $\bar{\ss}$ corresponds to a regular Dirichlet subspace of $(\EE,\FF)$. Now we shall divide the proof of Theorem~\ref{THM21} into two steps. The first one is to assert the regular Dirichlet subspace associated with $\bar{\ss}$ is bigger than the smallest closed extension $\bar{\FF}$ of $C_c^\infty(I)$. 

\begin{lemma}\label{LM32}
Let $\bar{\tt}$ be the inverse function of $\bar{\ss}$. Then $\bar{\tt}$ is absolutely continuous on $\bar{J}:=\bar{\ss}(I)$ and $\bar{\tt}'\in L^2_\mathrm{loc}(\bar{J})$. Particularly, $C_c^\infty(I)\subset \FF^{(\bar{\ss},m)}_0$. 
\end{lemma}
\begin{proof}
The absolute continuity of $\bar{\tt}$ is obvious from the fact $g>0$ a.e. (we refer this fact to the notes before Theorem~\ref{THM21}). 
	For any compact subset $J'\subset \bar{J}$, set $I':=\bar{\mathtt{t}}(J')$ which is also a compact subset of $I$. It follows from $({\bar{\mathtt{t}}}'\circ \bar{\mathtt{s}})\cdot \bar{\mathtt{s}}'=1$ and $\lambda_I(H)=0$ that
	\[\begin{aligned}
		\int_{J'} (\bar{\mathtt{t}}'(x))^2dx&=\int_{I'}(\bar{\mathtt{t}}'\circ \bar{\mathtt{s}})^2d\bar{\mathtt{s}}\\&=\int_{I'}\bar{\mathtt{t}}'(\bar{\mathtt{s}}(x))dx\\&=\int_{I'\setminus H} \frac{1}{\bar{\mathtt{s}}'(x)}dx\\&=\int_{I'\setminus H} \frac{1}{g(x)}dx.
	\end{aligned}\] 
	On the other hand, 
	\[
		{\frac{1}{g}}\bigg|_{I\setminus H}= {\frac{dx}{d\mathtt{s}}}\bigg|_{I\setminus H}={\mathtt{t}'\circ \mathtt{s}}|_{I\setminus H},\quad \lambda_I\text{-a.e.}
	\]
Since $\mathtt{s}(I')$ is compact in $J$ and $\mathtt{t}'\in L^2_{\text{loc}}(J)$ by Lemma~\ref{LM21}, we may conclude
	\[
		\int_{J'} (\bar{\mathtt{t}}'(x))^2dx=\int_{I'\setminus H} \mathtt{t}'\circ \mathtt{s}(x) dx=\int_{I'}(\mathtt{t}'\circ \mathtt{s})^2d\mathtt{s}=\int_{\mathtt{s}(I')}\mathtt{t}'(x)^2dx<\infty.
	\]
The second assertion is clear by using Lemma~\ref{LM21} again. 
That completes the proof. 
\end{proof}

\begin{proof}[Proof of Theorem~\ref{THM21}]
It follows from Lemma~\ref{LM32} that $\bar{\FF}\subset \FF^{(\bar{\ss},m)}_0$ and thus
\[
	(\bar{\EE},\bar{\FF})\prec \left(\EE^{(\bar{\ss},m)}, \FF^{(\bar{\ss},m)}_0\right),\quad \left(\EE^{(\bar{\ss},m)}, \FF^{(\bar{\ss},m)}_0\right)\prec (\EE,\FF). 
\]
Denote the scaling function of $(\bar{\EE},\bar{\FF})$ by $\hat{\ss}$. We only need to prove $\hat{\ss}=\bar{\ss}$. In fact, we have $d\hat{\ss}\ll d\bar{\ss}$ and 
	\[
		\frac{d\hat{\ss}}{d\bar{\ss}}= 0\text{ or } 1,\quad d\bar{\ss}\text{-a.e.}
	\]
On the other hand, $d\bar{\ss}=g\cdot \lambda_I$ with $g>0$ a.e., hence $d\bar{\ss}$ is equivalent to $\lambda_I$. If the $d\bar{\ss}$-measure of the set 
	\[
		Z_{d\hat{\ss}/d\bar{\ss}}:=\left\{x\in I: \frac{d\hat{\ss}}{d\bar{\ss}}(x)=0\right\}
	\]
is positive, then we can obtain $\lambda_I(Z_{d\hat{\ss}/d\bar{\ss}})>0$. Since $d\bar{\ss}$ is equivalent to $\lambda_I$, the zero point set of $\hat{\ss}'$
	\[
		Z_{\hat{\ss}'}:=\{x\in I:\hat{\ss}'(x)=0\}
	\]
equals $Z_{d\hat{\ss}/d\bar{\ss}}$ in the sense of a.e. (also $d\bar{\ss}$-a.e.). Thus $\lambda_I(Z_{\hat{\ss}'})>0$ and the inverse function 
	\[
		\hat{\tt}:=\hat{\ss}^{-1}	
	\]
is not absolutely continuous. However, it follows from $C_c^1(I)\subset \bar{\mathcal{F}}$ and Lemma~\ref{LM21} that $\hat{\tt}$ is absolutely continuous.  Therefore, we may conclude $d\bar{\ss}(Z_{d\hat{\ss}/d\bar{\ss}})=0$ and that implies $\hat{\ss}=\bar{\ss}$ (as we noted in \S\ref{SEC21}, any scaling function is forced to be $0$ at a fixed point $e\in I$).  
		
		The second assertion follows from the fact $(\bar{\mathcal{E}},\bar{\mathcal{F}})=(\mathcal{E,F})$ if and only if $\ss=\bar{\ss}$, in other words, $\ss$ is absolutely continuous. That completes the proof.
\end{proof}

\subsection{Cartesian product and skew product} 
The proofs of Theorem~\ref{THM25} and \ref{THM27} are somewhat similar to the one-dimensional cases. The essential factors are the scaling functions, whereas the details are a little different.

\subsubsection{Cartesian product}
Let us first prove the case of Cartesian product. All the notations are given in \S\ref{SEC221}. Note that for any $u,v\in \FF$, the form $\EE$ of Cartesian product may be written as
\[
	\EE(u,v)=\sum_{i=1}^d \int_{\hat{E}_i} \EE^i(u_{\hat{x}_i}(\cdot),v_{\hat{x}_i}(\cdot))\hat{m}_i(d\hat{x}_i),
\]
where $\hat{E}_i=I_1\times \cdots \times I_{i-1}\times I_{i+1} \times \cdots \times I_d$, $\hat{x}_i=(x_1,\cdots, x_{i-1},x_{i+1},\cdots, x_d)$, $\hat{m}_i$ is the corresponding product measure on $\hat{E}_i$, and $u_{\hat{x}_i}(\cdot):=u(x_1,\cdots, x_{i-1}, \cdot, x_{i+1}, \cdots, x_d)$.  We refer the detailed expression of $(\EE,\FF)$ to \cite[Theorem~1.4]{O89} or \cite[Proposition~3.1]{LY15}. 

\begin{proof}[Proof of Lemma~2.4]
In fact, $C_c^\infty(E)\subset \FF$ is equivalent to that any coordinate function $f^i\in \FF_\mathrm{loc}$, where $f^i(x)=x_i$ for any $x=(x_1,\cdots,x_d)$ and $1\leq i\leq d$. We refer this fact to \cite[Theorem~5.6.2 and Corollary~5.6.2]{FOT11}. Note that $f^i\in \FF_\mathrm{loc}$ if and only if $f^i_{\hat{x}_i}\in \FF^i_\mathrm{loc}$. From Lemma~\ref{LM21}, we may deduce that $f^i_{\hat{x}_i}\in \FF^i_\mathrm{loc}$ is equivalent to that $\mathtt{t}_i$ is absolutely continuous and $\tt'_i\in L^2(J_i)$. That completes the proof. 
\end{proof}

Then the proof of Theorem~\ref{THM25} is similar to that of Theorem~\ref{THM21} and not tough.

\begin{proof}[Proof of Theorem~\ref{THM25}]
Let $\bar{\tt}_i$ be the inverse of $\bar{\ss}_i$ for any $1\leq i\leq d$. 
Further let $(\bar{\EE}^i,\bar{\FF}^i)$ be the associated Dirichlet form of $\bar{\mathbf{X}}^i$ and $(\hat{\EE},\hat{\FF})$ that of the Cartesian product $(\bar{\mathbf{X}}^1,\cdots, \bar{\mathbf{X}}^d)$. We need to prove $\bar{\FF}=\hat{\FF}$. In fact, it follows from Lemma~\ref{LM32} that $\bar{\mathtt{t}}_i$ is absolutely continuous and $\bar{\tt}'_i\in L^2_\mathrm{loc}(\bar{J}_i)$, where $\bar{J}_i:=\bar{\ss}_i(I_i)$. Then Lemma~\ref{LM24} implies $C_c^\infty(E)\subset \hat{\FF}$, and hence $\bar{\FF}\subset \hat{\FF}$. 

On the contrary, since $C_c^\infty(I_i)$ is a special standard core of $(\bar{\EE}^i,\bar{\FF}^i)$, we may deduce that
\[
	\mathcal{C}:=C_c^\infty(I_1)\otimes \cdots \otimes C_c^\infty(I_d)
\]
is a core of $(\hat{\EE},\hat{\FF})$. However, clearly we have $\mathcal{C}\subset C_c^\infty(E)$. That indicates $\hat{\FF}\subset \bar{\FF}$. 

The last assertion is similar to Theorem~\ref{THM21}. That completes the proof. 
\end{proof}

\subsubsection{Skew product}

The proofs of Lemma~\ref{LM26} and Theorem~\ref{THM27} are very similar to those of Lemma~\ref{LM24} and Theorem~\ref{THM25}. 

\begin{proof}[Proof of Lemma~\ref{LM26}]
Assume $C_c^\infty(E)\subset \FF$. Since $C_c^\infty(I)\otimes C^\infty(S^{d-1})\subset C_c^\infty(E)$, we may deduce from \eqref{EQ22OMU} that $C_c^\infty(I)\subset \FF^1$. Then it follows from Lemma~\ref{LM21} that $\mathtt{t}$ is absolutely continuous and $\mathtt{t}'\in L^2_\mathrm{loc}(J)$. 

On the contrary, assume $\mathtt{t}$ satisfies the assumptions, equivalently by Lemma~\ref{LM21}, we have $C_c^\infty(I)\subset \FF^1$. It follows from $\FF^1\otimes \FF^2\subset \FF$ (Cf. \cite[Theorem~1.4]{O89}) that $C_c^\infty(I)\otimes C^\infty(S^{d-1})\subset \FF$. Denote the associated Dirichlet form of \eqref{EQ22XTR} by $(\hat{\EE},\hat{\FF})$. From \cite[Theorem~1.1]{FO89} and \cite[Theorem~1.4]{O89} (or \cite[Proposition~3.1]{LY15}), we may conclude that $C_c^\infty(I)\otimes C^\infty(S^{d-1})$ and $C_c^\infty(E)$ are both cores of $(\hat{\EE},\hat{\FF})$. Note that for any $u,v\in C_c^\infty(I)\otimes C^\infty(S^{d-1})$, we have 
\[
	\EE(u,v)=\hat{\EE}(u,v). 
\]
Hence $\hat{\FF}\subset \FF$, which implies $C_c^\infty(E)\subset \hat{\FF}\subset \FF$. That completes the proof. 
\end{proof}

From the fact $C_c^\infty(I)\otimes C^\infty(S^{d-1})$ and $C_c^\infty(E)$ are both cores of $(\hat{\EE},\hat{\FF})$, it follows that the proof of Theorem~\ref{THM27} is actually obvious. So we omit its details. 

\section{Fukushima's decomposition}\label{SEC4}

In this section, we shall consider the Fukushima's decompositions (Cf. \cite{Fu79, Fu81, OY84} and also \cite{FOT11, CF12}) under our basic assumption \eqref{EQ1FFL} (or equivalently, \eqref{EQ2CCE} by \S\ref{SEC3}). 

\subsection{One-dimensional case}\label{SEC41}

We first explore one-dimensional case. Let $\mathbf{X}$ be the diffusion process in \S\ref{SEC21}, i.e. a minimal diffusion process on $I$ with the scaling function $\ss$ and speed measure $m$. Its associated Dirichlet form is denoted by $(\EE,\FF)$ with the expression \eqref{EQ2FFS}. Further denote all locally martingale additive functionals of finite energy by $\overset{\circ}{\mathcal{M}}_\mathrm{loc}$ and all locally continuous additive functionals of zero energy by $\mathcal{N}_{c,\mathrm{loc}}$. We refer their detailed definitions to \cite[Chapter 5]{FOT11}. Since the coordinate function $f\in \FF_\mathrm{loc}$, it follows from \cite[Theorem~5.5.1]{FOT11} that there uniquely exist an $M^{[f]}\in \overset{\circ}{\mathcal{M}}_\mathrm{loc}$ and an $N^{[f]}\in \mathcal{N}_{c,\mathrm{loc}}$ such that
\begin{equation}\label{EQ41FXT}
	f(X_t)-f(X_0)=M^{[f]}_t+N^{[f]}_t,\quad t\geq 0, \ \mathbf{P}^x\text{-a.s.}, \ \forall x\in I. 
\end{equation}
Hereafter, if not otherwise stated, this kind of equalities hold in the sense of $\mathbf{P}^x$-a.s. for any $x\in I$. 

Next, we shall characterize $M^{[f]}$ and $N^{[f]}$ via the essential determiners $\ss$ and $m$ of $\mathbf{X}$. 
Since $\mathbf{X}$ is a diffusion process, naturally $M^{[f]}$ is also a continuous additive functional. Note that $M^{[f]}$ is uniquely determined by its sharp bracket $\langle M^{[f]}\rangle$, which is a PCAF. Furthermore, $\langle M^{[f]}\rangle$ is also uniquely determined by its Revuz measure $\mu_{\langle M^{[f]}\rangle}$. In the following lemma, we shall give an expression of $\mu_{\langle M^{[f]}\rangle}$ by the scaling function $\ss$. 

\begin{lemma}\label{LM41}
	Assume \eqref{EQ1FFL} is satisfied. Then the Revuz measure of the sharp bracket $\langle M^{[f]}\rangle$ is 
	\[
		\mu_{\langle M^{[f]}\rangle}=\tilde{m},
	\]
where $\tilde{m}$ is given by \eqref{EQ24MXT}. 
\end{lemma}
\begin{proof}
For any relatively compact open subset $I'$ of $I$, take a  bounded function $f^{I'}\in \mathcal{F}$ such that $f^{I'}=f$ on $I'$. We may also write the Fukushima's decomposition with respect to $f^{I'}$, and its MAF part is denoted by $M^{[f^{I'}]}$. Note that the Revuz measure $\mu_{\langle M^{[f^{I'}]} \rangle}$ equals the energy measure $\mu_{\langle f^{I'}\rangle}$ (Cf. \cite[Theorem~5.2.3]{FOT11}).  By \cite[Theorem~5.2.3]{FOT11} again, we have
	\[
		\int_I u(x)\mu_{\langle f^{I'}\rangle}(dx)=2\mathcal{E}\left(u\cdot f^{I'}, f^{I'}\right)-\mathcal{E}\left(\left(f^{I'}\right)^2,u\right)
	\]
for any $u\in \mathcal{F}\cap C_c(I)$. Since
	\[
	\begin{aligned}
		2\mathcal{E}\left(u\cdot f^{I'}, f^{I'}\right)-\mathcal{E}\left(\left(f^{I'}\right)^2,u\right)&=\int_I \frac{d(u\cdot f^{I'})}{d\ss}\frac{df^{I'}}{d\ss}d\ss-\frac{1}{2}\int_I \frac{d(f^{I'})^2}{d\ss}\frac{du}{d\ss}d\ss  \\
			&=\int_I u\cdot\left(\frac{d f^{I'}}{d\ss}\right)^2 d\ss,
	\end{aligned}
	\]
we may deduce from the regularity of $(\EE,\FF)$  that
	\[
		\int_{I'} u(x)\mu_{\langle f^{I'}\rangle}(dx)=\int_{I'} u\cdot \left(\frac{d f^{I'}}{d\ss}\right)^2 d\ss,\quad u\in C_c(I').
	\]
Then it follows from \cite[Theorem~5.5.2]{FOT11} that
	\[
		\mu_{\langle M^{[f]}\rangle}\big|_{I'}=\mu_{\langle f^{I'}\rangle}\big|_{I'}=\left(\frac{d f^{I'}}{d\ss}\right)^2 d\ss\bigg|_{I'}= \tt'(\ss)^2 d\ss|_{I'}.
	\]
From the fact $I'$ is arbitrary and $dx=\tt'(\ss)d\ss$, we obtain
\[
	\mu_{\langle M^{[f]}\rangle}=\tt'(\ss)^2d\ss=(\tt'\circ\ss)(x)dx=\tilde{m}.
\]
That completes the proof. 
\end{proof}
\begin{remark}
From the lemma above, we know that the assumption $\tilde{m}=m$ implies
\[
	\langle M^{[f]} \rangle_t=t\wedge \zeta,\quad t\geq 0, 
\]
where $\zeta$ is the lifetime of $\mathbf{X}$. 
Thus $M^{[f]}$ is equivalent to a Brownian motion before $\zeta$. 
\end{remark}

In what follows, we shall characterize the zero energy part $N^{[f]}$ in \eqref{EQ41FXT}, but some conceptions should be prepared at first. An additive functional $A$ is said to be of bounded variation if $A_t(\omega)$ is of bounded variation in $t$ on each compact subinterval of $[0,\zeta(\omega))$ for every fixed $\omega$ in the defining set of $A$. It is known that a CAF $A$ is of bounded variation if and only if $A$ can be expressed as a difference of two PCAF's:
\[
	A_t(\omega)=A^1_t(\omega)-A^2_t(\omega),\quad t<\zeta(\omega), 
\]
where $A^1,A^2\in \mathbf{A}^+_c$ and $\mathbf{A}_c^+$ is the set of all PCAFs. Let $\mu_{A^1}$ and $\mu_{A^2}$ be the Revuz measures associated with $A^1$ and $A^2$ respectively. Then
\[
	\mu_A:=\mu_{A^1}-\mu_{A^2}
\]
is called the smooth signed measure of $A$. Note that every Revuz measure with respect to one-dimensional diffusion process $\mathbf{X}$ is Radon (Cf. \cite[Chapter X, Proposition~2.7]{RY99}. Thus the smooth signed measure of $A$ is Radon.

Generally, a measure $\nu$ on $I$ is called a smooth signed measure if there exists a generalized nest $\{F_k: k\geq 1\}$ (Cf. \cite[\S2.2]{FOT11}) such that $1_{F_k}\cdot \nu$ is a finite signed Borel measure charging no set of zero capacity for each $k$, and further $\nu(I\setminus \cup_{k\geq 1} F_k)=0$. Such a generalized nest $\{F_k:k\geq 1\}$ is said to be associated with $\nu$. All the conceptions above may be also extended to local additive functionals, and we refer them to \cite[\S5.5]{FOT11}. 

Since any singleton of $I$ is of positive capacity with respect to $\mathbf{X}$, it follows that
\[
	I=\bigcup_{k\geq 1} F_k, 
\]
for any generalized nest $\{F_k:k\geq 1\}$. 
Furthermore, for a closed subset $F\subset I$, we put
\[
\begin{aligned}
	\FF_F&:=\{u\in \FF: \tilde{u}=0\text{ q.e. on }I\setminus F\}, \\
	\FF_{b,F}&:=\{u\in \FF_b: \tilde{u}=0\text{ q.e. on }I\setminus F\},
\end{aligned}\]
where $\FF_b$ is all bounded functions in $\FF$. Note that since $\mathbf{X}$ is a one-dimensional irreducible diffusion process, the quasi-continuous version $\tilde{u}$ equals $u$ and ``q.e." is indeed ``everywhere". 

\begin{lemma}\label{LM42}
The CAF $N^{[f]}$ in \eqref{EQ41FXT} is of bounded variation, if and only if there exists a smooth signed measure $\nu$ such that for any compact set $K$,
\begin{equation}\label{EQ42EFU}
	\EE(f, u)=\langle \nu, u\rangle, \quad u\in \FF_{b, K}.
\end{equation}
Particularly, the smooth signed measure of $N^{[f]}$ equals $-\nu$, i.e. $\mu_{N^{[f]}}=-\nu$.
\end{lemma}
\begin{proof}
By \cite[Theorem 5.5.4]{FOT11}, $N^{[f]}$ is of bounded variation if and only if there exists a smooth signed measure $\nu$ such that 
	\begin{equation}\label{EFU}
		\mathcal{E}(f, u)=\langle \nu, u\rangle,\quad u\in \bigcup_{k=1}^\infty\mathcal{F}_{b, F_k}
	\end{equation}
for a generalized nest $\{F_k: k\geq 1\}$ of increasing compact sets associated with $\nu$, i.e. $|\nu|(F_k)<\infty$ for each $k$. Particularly, $\mu_{N^{[f]}}=-\nu$.

Clearly, \eqref{EQ42EFU} implies \eqref{EFU}. Thus we need only to prove the necessity of \eqref{EQ42EFU}. Fix a compact set $K$ and a function $u\in \FF_{b,K}$. Take a bounded open interval $(a',b')$ such that 
\[
	K\subset (a',b')\subset [a',b']\subset I,
\]
and a function $\psi\in C_c^\infty(I)$ such that $0\leq \psi\leq 1$, $\psi=1$ on $K$ and $\text{supp}[\psi]\subset (a',b')$. Clearly, $\psi\in \FF_b$. Since $\{F_k:k\geq 1\}$ is a generalized nest, equivalently, $\cup_{k\geq 1}\FF_{F_k}$ is $\EE_1$-dense in $\FF$, thus we may take a sequence of functions $\{u_n:n\geq 1\}\subset \cup_{k\geq 1}\FF_{F_k}$ such that
\begin{equation}\label{EQ42UNU}
	\|u_n-u\|_{\EE^1}\rightarrow 0,\quad n\rightarrow \infty.
\end{equation}
Note that $\{u_n: n\geq 1\}$ could be taken to be uniformly bounded (Cf. \cite[Theorem~1.4.2 (iii)]{FOT11}). Without loss of generality, we may also assume $u_n\rightarrow u$ pointwisely as $n\rightarrow \infty$. Let 
\[
	v_n:=u_n\cdot \psi.
\]
We assert $\|v_n-u\|_{\EE_1}\rightarrow 0$ as $n\rightarrow 0$. In fact, $\|v_n-u\|_m\leq \|u_n-u\|_m\rightarrow 0$ as $n\rightarrow \infty$. On the other hand, 
\[
\begin{aligned}
	\EE(v_n-u,v_n-u)&=\frac{1}{2}\int_I \left( \frac{d(v_n-u)}{d\ss}\right)^2d\ss  \\
	&=\frac{1}{2}\int_I \left( \frac{d\left((u_n-u)\cdot \psi\right)}{d\ss}\right)^2d\ss  \\
	&\leq \int_I \left(\left(\frac{d(u_n-u)}{d\ss}\right)^2\cdot \psi^2 +(u_n-u)^2\left(\frac{d\psi}{d\ss}\right)^2  \right)d\ss
\end{aligned}\]
Since $\psi$ is bounded, it follows from \eqref{EQ42UNU} that 
\[
	\int_I \left(\frac{d(u_n-u)}{d\ss}\right)^2\cdot \psi^2  d\ss\leq 2\EE(u_n-u,u_n-u)\rightarrow 0,\quad n\rightarrow \infty.
\]
Note that $(d\psi/d\ss)^2d\ss=\left((\psi')^2\cdot (\tt'\circ \ss)^2\right)d\ss=(\psi')^2\cdot \tilde{m}$ is a finite measure supported on $(a',b')$. From \cite[(2.2.32)]{CF12}, we may deduce that there exists a constant $C$ (only depends on $(a',b')$) such that 
\[
\sup_{x\in (a',b')} \left(u_n(x)-u(x)\right)^2\leq C\cdot \EE_1(u_n-u,u_n-u). 
\]
That implies 
\[
	\int_I (u_n-u)^2\left(\frac{d\psi}{d\ss}\right)^2d\ss\leq C\cdot \EE_1(u_n-u,u_n-u)\cdot \int_{a'}^{b'}\left(\frac{d\psi}{d\ss}\right)^2d\ss \rightarrow 0
\]
as $n\rightarrow \infty$. 

Clearly, $v_n\in \cup_{k\geq 1} \FF_{F_k}$, and it follows from \eqref{EFU} that $\EE(f, v_n)=\langle \nu, v_n\rangle$. Since $\|v_n-u\|_{\EE_1}\rightarrow 0$ as $n\rightarrow \infty$, we have $\lim_{n\rightarrow \infty} \EE(f, v_n)=\EE(f,u)$. Furthermore, $v_n$ is uniformly bounded and supported on $(a',b')$, and $v_n\rightarrow u$ pointwisely. Note that $\nu$ is signed Radon (hence a finite measure if it is restricted on $(a',b')$). Thus we can also obtain $\lim_{n\rightarrow \infty} \langle \nu, v_n\rangle=\langle \nu, u\rangle$. Hence
\[
	\EE(f,u)=\langle \nu, u\rangle.
\]
That completes the proof.
\end{proof}

The following proposition gives another characterization equivalent to the fact $N^{[f]}$ is of bounded variation via the scaling function $\ss$. 

\begin{proposition}\label{PRO42}
Assume \eqref{EQ1FFL} is satisfied, and let $N^{[f]}$ be the local CAF of zero energy in \eqref{EQ41FXT}. Then $N^{[f]}$ is of bounded variation, if and only if an a.e. version of $\tt'$ is locally of bounded variation on $J$. Furthermore, if $N^{[f]}$ is of bounded variation, then the smooth signed measure of $N^{[f]}$ is
\begin{equation}\label{EQ41MNF}
\mu_{N^{[f]}}=\frac{1}{2}d\tt_*(\tt'),
\end{equation}
where $d\tt_*(\tt')$ is the image measure of $d\tt'$ with respect to the transform $\tt: J\rightarrow I$. 
\end{proposition}
\begin{proof}
We first assume $N^{[f]}$ is of bounded variation. Let $\nu$ be the smooth signed measure in Lemma~\ref{LM42}. For any function $u\in \mathcal{F}_b$ with compact support $K$, there exists an absolutely continuous function $\phi$ on $J$ with $\text{supp}[\phi]\subset \ss(K)$ and $\phi'\in L^2(J)$ such that $u=\phi\circ \ss$. It follows that
	\begin{equation}\label{MEF}
		\mathcal{E}(f,u)=\frac{1}{2}\int_I \frac{df}{d\ss}\frac{du}{d\ss}d\ss=\frac{1}{2}\int_J \tt'(x)\phi'(x)dx
	\end{equation}
	and 
	\[
		\langle \nu, u\rangle=\int_I \phi\circ \ss(x) \nu(dx)=\int_{J}\phi(y) \ss_*(\nu)(dy).
	\]
	Since $\ss_*(\nu)$ (i.e. the image measure of $\nu$ via the transform $\ss$) is a Radon measure on $J$, there exists a function of locally bounded variation, say $F$, on $J$ such that
	\[
		\ss_*(\nu)=dF,
	\]
	where $dF$ is the Lebesgue-Stieltjes measure with respect to $F$. 
	Since the support of $u$ (hence of $\phi$) is compact, we may deduce
	\begin{equation}\label{NTI}
	\langle \nu, u\rangle=\int_J \phi(x)dF(x)=-\int_J F(x)\phi'(x)dx,
	\end{equation}
It follows from Lemma~\ref{LM42}, \eqref{MEF} and \eqref{NTI} that
	\begin{equation}\label{IJQ}
		\int_J (\tt'(x)+2F(x))\phi'(x)dx=0
	\end{equation}
	for any $\phi\in C_c^\infty (J)$. Then from \cite[Corollary~3.32]{AF03}, we may conclude
	\[
		\tt'+2F=C,\quad \text{a.e. on }J
	\]
	for some constant $C$. Clearly, $\tt'$ is of locally bounded variation and $\ss_*(\nu)=-1/2d\tt'$. Hence
	\begin{equation}\label{EQ41NTT}
		\nu=-\frac{1}{2}d\tt_*(\tt').
	\end{equation}
	
On the contrary, assume that $\tt'$ is of locally bounded variation. For any $u\in \mathcal{F}_b$ with compact support, we may take an absolutely continuous function $\phi$ on $J$ with compact support and $\phi'\in L^2(J)$ such that $u=\phi\circ \ss$. Thus we have
	\[
		\mathcal{E}(f,u)=\frac{1}{2}\int_J \tt'(x)\phi'(x)dx=-\frac{1}{2}\int_J \phi(x)d\tt'(x)=-\frac{1}{2}\int_I u(x)d\tt_*(\tt')(x).
	\]
Clearly, $\frac{1}{2}d\tt_*(\tt')$ is a smooth signed Radon measure and \eqref{EQ42EFU} holds for $\nu=-\frac{1}{2}d\tt_*(\tt')$. From Lemma~\ref{LM42}, we conclude $N^{[f]}$ is of bounded variation. That completes the proof.
\end{proof}

The condition of $\tt$ in Proposition~\ref{PRO42} looks a little awkward. Particularly, if $\tt\in C^1(J)$ and $\tt'$ is absolutely continuous (i.e. \textbf{(H4')} in \S\ref{SEC24} is satisfied), then $\tt$ satisfies Proposition~\ref{PRO42} naturally. However, we may also give some examples, in which $\tt'$ is only of bounded variation, but not absolutely continuous, see Example~\ref{EXA53}. Furthermore, \textbf{(H4')} owns the following characterization. 

\begin{corollary}\label{COR44}
Assume \eqref{EQ1FFL} is satisfied, and $N^{[f]}$ is of bounded variation. Then \textbf{(H4')} holds, i.e. $\tt\in C^1(J)$ and $\tt'$ is absolutely continuous, if and only if $\mu_{N^{[f]}}\ll \tilde{m}$. 
\end{corollary}
\begin{proof}
In the notes between Lemma~\ref{LM214} and Corollary~\ref{COR215}, we already illustrated that \textbf{(H4')} implies $\mu_{N^{[f]}}\ll \tilde{m}$. On the contrary, since $\mu_{N^{[f]}}= 1/2\cdot d\tt_*(\tt')=(1/2\cdot d\tt') \circ \ss$ by Proposition~\ref{PRO42}, it follows that $\mu^{N^{[f]}}\ll \tilde{m}$ implies $d\tt'\ll \tt'(x)^2dx$. Clearly, $\tt'(x)^2dx\ll dx$, and hence $d\tt'\ll dx$. That means $\tt'$ is absolutely continuous and completes the proof. 
\end{proof}

\subsection{Proofs of \S\ref{SEC24}}

We first prove Lemma~\ref{LM212}, i.e. $\mathfrak{D}$ has uncountable elements. 

\begin{proof}[Proof of Lemma~\ref{LM212}]
Since $\ss$ is not absolutely continuous, we may write the Lebesgue decomposition of $d\ss$ as \eqref{EQ2SGL}, i.e. 
\[
	d\ss=g\cdot \lambda+\kappa,\quad \kappa\perp \lambda,
\]
where $\lambda$ is the Lebesgue measure on $\mathbb{R}$. 
Let $H$ be the set of $\lambda$-zero measure such that $\kappa(\mathbb{R}\setminus H)=0$. Clearly, $\kappa(H)>0$. 

Set $k(x):=\kappa([-x,x])$ for any $x\geq 0$. Since $\ss$ is continuous, it follows that any singleton is of zero $\kappa$-measure, and thus $k$ is increasing and continuous. Fix arbitrary constant $0<c<\kappa(H)=k(\infty)$, we may take a constant $x_c> 0$ such that $k(x_c)=\kappa([-x_c,x_c])=c$. Set
\[
	\ss_c(x):=\int_0^x g(y)dy+\int_0^x 1_{[-x_c,x_c]}(y)\kappa(dy)
\]
for any $x\in \mathbb{R}$. Clearly, $\ss_c$ is strictly increasing and continuous, i.e. a scaling function. Let $(\EE^c,\FF^c)$ be the associated Dirichlet form of the minimal diffusion process with scaling function $\ss_c$ and speed measure $m$. Furthermore, 
\[
	\frac{d\ss_c}{d\ss}=1_{H^c\cup [-x_c,x_c]},\quad d\ss\text{-a.e.}, 
\]
and 
\[
	\frac{d\bar{\ss}}{d\ss_c}=1_{H^c},\quad d\ss_c\text{-a.e.},
\]
where $d\bar{\ss}=g\cdot \lambda$ is the absolutely continuous part of $\ss$. Note that $d\ss(H\cap \{y: |y|>x_c\})=\kappa (H\cap \{y: |y|>x_c\})>0$ and $d\ss_c(H)=\kappa(H\cap [-x_c,x_c])>0$. They imply $(\bar{\EE},\bar{\FF})$ is a proper regular Dirichlet subspace of $(\EE^c,\FF^c)$, and $(\EE^c,\FF^c)$ is a proper regular Dirichlet subspace of $(\EE,\FF)$ (Cf. \cite[Theorem~4.1]{FHY10}. Particularly, $(\EE^c,\FF^c)\in \mathfrak{D}$. 

For two different constants $c_1,c_2\in (0, \kappa(H))$, we may prove similarly that if $c_1<c_2$, then $(\EE^{c_1},\FF^{c_1})$ is a proper regular Dirichlet subspace of $(\EE^{c_2}, \FF^{c_2})$. Hence $\mathfrak{D}$ has uncountable elements. That completes the proof. 
\end{proof}

\begin{remark}
From the above proof, we may also deduce that
\[
	[0, \kappa(\mathbb{R})]\rightarrow \mathfrak{D},\quad c\mapsto (\EE^c,\FF^c)
\] 
is a one-to-one correspondence. Particularly, $0$ corresponds to the smallest element $(\bar{\EE},\bar{\FF})$ of $\mathfrak{D}$, and $\kappa(\mathbb{R})(\leq \infty)$ corresponds to the biggest one $(\EE,\FF)$. 
\end{remark}

Now we turn to prove Theorem~\ref{THM214}. The proof will be divided into several steps. Take a Dirichlet form $(\EE^c,\FF^c)\in \mathfrak{D}$. Its associated diffusion process is denoted by $\mathbf{X}^c=(X^c_t, \mathbf{P}^x_c)_{x\in \mathbb{R}}$. Since $C_c^\infty(\mathbb{R})\subset \FF^c$, it follows that the coordinate function $f\in \FF^c_\mathrm{loc}$. Thus we may write the Fukushima's decomposition of $\mathbf{X}^c$ with respect to $f$, i.e.
\begin{equation}\label{EQ42FXC}
	f(X^c_t)-f(X^c_0)=M^{[f],c}_t+N^{[f],c}_t,\quad t\geq 0,
\end{equation}
where $M^{[f],c}\in \overset{\circ}{\mathcal{M}}^c_\mathrm{loc}, N^{[f],c}\in \mathcal{N}^c_{c,\mathrm{loc}}$, $\overset{\circ}{\mathcal{M}}^c_\mathrm{loc}$ and $\mathcal{N}^c_{c,\mathrm{loc}}$ stand for the sets of local MAFs and local CAFs of zero energy with respect to $\mathbf{X}^c$. Clearly, $M^{[f],c}$ is also continuous MAF, and we denote the Revuz measure of its sharp bracket $\langle M^{[f],c} \rangle$ by $\mu^c_{\langle f\rangle}$, i.e.
\[
	\mu^c_{\langle f\rangle}:= \mu_{\langle M^{[f],c} \rangle}. 
\]
We also write $\mu_{\langle M^{[f]}\rangle}$ in Lemma~\ref{LM41} as $\mu_{\langle f \rangle}$ for short. 

\begin{lemma}\label{LM45}
The Revuz measure of the sharp bracket $\langle M^{[f],c}\rangle$ equals that of $\langle M^{[f]}\rangle$, in other words,
\[
	\mu^c_{\langle f\rangle}=\mu_{\langle f\rangle}. 
\]
\end{lemma}
\begin{proof}
Note that $\mu^c_{\langle f\rangle}$ and $\mu_{\langle f\rangle}$ are both Radon measures on $\mathbb{R}$. On the other hand, it follows from \cite[Theorem~5.2.3]{FOT11} and $(\EE^c,\FF^c)\prec (\EE,\FF)$ that for any $u\in C_c^\infty(\mathbb{R})$, 
\[
	\int_\mathbb{R}u(x)\mu_{\langle f\rangle}^c(dx)=2\EE^c(u\cdot f,f)-\EE^c(f^2, u)=2\EE(u\cdot f,f)-\EE(f^2, u)=\int_\mathbb{R}u(x)\mu_{\langle f\rangle}(dx). 
\]
Thus we can obtain that $\mu^c_{\langle f\rangle}=\mu_{\langle f\rangle}$. That completes the proof.
\end{proof}

\begin{lemma}\label{LM46}
If $N^{[f]}$ is of bounded variation, then $N^{[f],c}$ is also of bounded variation. Furthermore, their smooth signed measures equals, i.e. 
\[
	\mu_{N^{[f]}}=\mu_{N^{[f],c}}. 
\]
\end{lemma}
\begin{proof}
	We first prove $N^{[f],c}$ is also of bounded variation. Take a generalized nest $\{K_n:n\geq 1\}$ of increasing compact sets with respect to $(\EE^c, \FF^c)$. From Lemma~\ref{LM42}, we may deduce that for any $u\in \FF^c_{b,K_n}\subset \FF_{b,K_n}$, 
	\[
		\EE^c(f,u)=\EE(f,u)=-\langle \mu_{N^{[f]}}, u\rangle.
	\]
Clearly, $-\mu_{N^{[f]}}$ is a smooth signed measure associated with $\{K_n: n\geq 1\}$ with respect to $\mathbf{X}^c$.
Then it follows from \cite[Theorem~5.5.4]{FOT11} that $N^{[f],c}$ is of bounded variation. Particularly, 
\[
\mu_{N^{[f],c}}=\mu_{N^{[f]}}=\frac{1}{2}d\tt_*(\tt').
\]
That completes the proof. 
\end{proof}

\begin{proof}[Proof of Theorem~\ref{THM214}] 
Let $(\Omega, \mathcal{G}, \mathbf{Q}^x)_{x\in \mathbb{R}}$ be the probability measure space of $\mathbf{X}^c$ with the adapted filtration $(\mathcal{G}_t)$. Fix an initial distribution $\xi$ on $\mathbb{R}$. 
It follows from \textbf{(H4)}, Proposition~\ref{PRO42} and Lemma~\ref{LM46} that $N^{[f],c}$ is of bounded variation and
\[
	\mu_{N^{[f],c}}(dx)=b(x)m(dx),
\]
where $b\in L^1_\mathrm{loc}(\mathbb{R},m)$ is given by \eqref{EQ24BDT}. By the uniqueness of the Revuz measure, we have
\begin{equation}\label{EQ42NFC}
	N^{[f],c}_t=\int_0^t b(X^c_s)ds,\quad t\geq 0,\ \mathbf{Q}^\xi\text{-a.s.}
\end{equation}
On the other hand, from \textbf{(H3)}, Lemma~\ref{LM41} and Lemma~\ref{LM45}, we may deduce that
\[
	\mu^c_{\langle f\rangle}(dx)=\sigma^2(x)m(dx). 
\]
Thus 
\begin{equation}\label{EQ42MCF}
	\langle M^{[f],c}\rangle_t=\int_0^t \sigma^2(X^c_s)ds,\quad t\geq 0, \ \mathbf{Q}^\xi\text{-a.s.}
\end{equation}

Now we carry everything to the enlargement $(\tilde{\Omega}, \tilde{\mathcal{G}}, \tilde{\mathbf{Q}}^\xi)$ of $(\Omega, \mathcal{G}, \mathbf{Q}^\xi)$. That is to take another probability measure space $(\Omega', \mathcal{G}', \mathbf{Q}')$ with some filtration and let $\tilde{\Omega}:=\Omega\times \Omega'$, $\tilde{\mathcal{G}}:=\mathcal{G}\times \mathcal{G}'$ and $\tilde{\mathbf{Q}}^\xi:=\mathbf{Q}^\xi\times \mathbf{Q}'$. Further set $(\tilde{\mathcal{G}}_t)_{t\geq 0}$ to be the induced filtration of the enlargement. We refer its detailed description to \cite[Chapter V \S1]{RY99}. Particularly, for any $\tilde{\omega}=(\omega, \omega')\in \tilde{\Omega}$, let
\[
\tilde{X}^c_t(\tilde{\omega}):=X^c_t(\omega), \quad \tilde{M}_t^{[f],c}(\tilde{\omega}):=M_t^{[f],c}(\omega),\quad  \tilde{N}_t^{[f],c}(\tilde{\omega}):=N_t^{[f],c}(\omega).
\]
They are all $(\tilde{\mathcal{G}}_t)$-adapted. On the other hand, \eqref{EQ42NFC} and \eqref{EQ42MCF}  are still right if we replace $X^c, M^{[f],c}, N^{[f],c}, \mathbf{Q}^\xi$ by $\tilde{X}^c, \tilde{M}^{[f],c}, \tilde{N}^{[f],c}, \tilde{\mathbf{Q}}^\xi$. 

Furthermore, since $d\langle M^{[f],c}\rangle_t \ll dt$, $\mathbf{Q}^\xi$-a.s., it follows from \cite[Chapter V, Theorem~3.9]{RY99} that there exists a $\tilde{\mathcal{G}}_t$-adapted standard Brownian motion $\tilde{B}=(\tilde{B}_t)_{t\geq 0}$ on $(\tilde{\Omega}, \tilde{\mathcal{G}}, \tilde{\mathbf{Q}}^\xi)$ such that 
\[
	\tilde{M}^{[f],c}_t=\int_0^t \sigma\left(\tilde{X^c_s}\right)d\tilde{B}_s,\quad t\geq 0, \ \tilde{\mathbf{Q}}^\xi\text{-a.s.}
\]
Then from \eqref{EQ42FXC}, we may obtain that
\[
	\tilde{X}^c_t-\tilde{X}^c_0=\int_0^t \sigma\left(\tilde{X}^c_s\right)d\tilde{B}_s+\int_0^t b(\tilde{X}^c_s)ds, \quad t\geq 0, \ \tilde{\mathbf{Q}}^\xi\text{-a.s.}
\]
Note that $(\tilde{X}^c_t)_{t\geq 0}$ is actually equivalent to $(X^c_t)_{t\geq 0}$. That completes the proof. 
\end{proof}

We need to point out in Theorem~\ref{THM214} the Brownian motion cannot be constructed directly on the same probability measure space as $\mathbf{X}^c$, since the lifetime $\zeta^c$ of $\mathbf{X}^c$ is probably finite.  Consequently, the enlargement of the probability measure space is necessary. For example, in the case of Corollary~\ref{COR215}, $\tilde{m}=m$, Lemma~\ref{LM41} and Lemma~\ref{LM45} indicate $\mu_{\langle f\rangle}^c=m$. Thus we have 
\[
	\langle M^{[f],c}\rangle_t=t\wedge \zeta^c,\quad t\geq 0, \quad \mathbf{Q}^\xi\text{-a.s.}.
\]
Then $M^{[f],c}$ is almost a standard Brownian motion. However, if $\zeta^c=\infty$, $\mathbf{Q}^\xi$-a.s. fails, then we still need to construct a standard Brownian motion $(B_t)_{t\geq 0}$ via the enlargement of $(\Omega, \mathcal{G}, \mathbf{Q}^\xi)$ (see \cite[Chapter V, Theorem~1.7]{RY99}) such that $M^{[f],c}_t=B_t$ for $t\leq \zeta^c$. Naturally, one may expect the enlargement could be dropped once providing $\mathbf{Q}^\xi(\zeta^c=\infty)=1$. Sometimes the answer is yes. 

\begin{lemma}
Let $(\Omega, \mathcal{G}, \mathbf{Q}^x)_{x\in \mathbb{R}}$ be the probability measure space of $\mathbf{X}^c$ with the adapted filtration $(\mathcal{G}_t)_{t\geq 0}$. Fix an initial distribution $\xi$ on $\mathbb{R}$. If $\mathbf{Q}^\xi(\zeta^c=\infty)=1$ and $m$ is absolutely continuous with respect to the Lebesgue measure (equivalently, \textbf{(H3')} is satisfied), then there exists a $\mathcal{G}_t$-adapted Brownian motion $(B_t)_{t\geq 0}$ such that 
\[
	M^{[f],c}_t=\int_0^t \sigma(X^c_s)dB_s,\quad t\geq 0, \ \mathbf{Q}^\xi\text{-a.s.}
\]
\end{lemma}
\begin{proof}
By \cite[Chapter V, Proposition~3.8]{RY99}, we only need to prove
\begin{equation}\label{EQ42DMF}
	d\langle M^{[f],c}\rangle_t =\sigma^2(X^c_t)dt
\end{equation}
is $\mathbf{Q}^\xi$-a.s. equivalent to the Lebesgue measure. Clearly, \eqref{EQ42DMF} is absolutely continuous with respect to the Lebesgue measure, so it suffices to prove for $\mathbf{Q}^\xi$-a.s. $\omega$, 
\[
	t\rightarrow \sigma^2(X^c_t(\omega))
\]
is a.e. strictly positive. Let $A:=\{x\in \mathbb{R}: \sigma(x)=0\}$. Note that $A$ is defined in the sense of a.e., since $\mathtt{t}$ is differential a.e., and its Lebesgue measure $|A|$ equals 
\[
|A|=\int_A dx=\int_{\{y: \tt'\circ\ss(y)=0\}} \tt'\circ \ss d\ss=0. 	
\] 
Put all the points at where $\tt'\circ \ss$ or $h$ do not exist (also of zero Lebesgue measure) into $A$, where $h$ is the density function of $m$ with respect to the Lebesgue measure. Then it still holds $|A|=0$, and moreover, 
\begin{equation}\label{EQ42TXC}
	\{t: \sigma^2(X^c_t)\text{ is not strictly positive}\}\subset \{t: X^c_t\in A\}.
\end{equation}
That also implies $m(A)=0$. 
Since the probability transition semigroup $P_t(x,dy)$ of $\mathbf{X}^c$ has a density function $p_t(x,y)$ with respect to $m$, we may deduce that
\[
	\mathbf{Q}^\xi\left(\int_0^\infty 1_{A}(X^c_t)dt\right)=\int_0^\infty dt\int_\mathbb{R}\xi(dx) \int_{\mathbb{R}} p_t(x,y)1_A(y)m(dy)=0. 
\]
Thus, $\int_0^\infty1_{A}(X^c_t)dt=0$,  $\mathbf{Q}^\xi$-a.s. In other words, the Lebesgue measure of $\{t: X^c_t\in A\}$ is zero. From \eqref{EQ42TXC}, we can obtain $t\rightarrow \sigma^2(X^c_t(\omega))$ is a.e. strictly positive. That completes the proof. 
\end{proof}

\begin{remark}\label{RM49}
If $\mathbf{Q}^x(\zeta^c=\infty)=1$ for any $x\in \mathbb{R}$, then $\mathbf{X}^c$ is said to be conservative. Clearly, the conservativeness of $\mathbf{X}^c$ indicates 
\[
	\mathbf{Q}^\xi(\zeta^c=\infty)=1,
\]
for any initial distribution $\xi$. On the other hand, it is known that $\mathbf{X}^c$ is conservative, if and only if neither of the boundary points is approachable in finite time. The latter assertion can also be characterized by the scaling function $\ss_c$ and speed measure $m$. We refer the details to \cite[Example~3.5.7]{CF12}. Particularly, if $c_1,c_2\in [0, \kappa(\mathbb{R})]$ and $c_1<c_2$, then the conservativeness of $\mathbf{X}^{c_1}$ indicates the conservativeness of $\mathbf{X}^{c_2}$. 
\end{remark}

Finally, let us assert the accessible assumptions \textbf{(H3')} and \textbf{(H4')} are stronger than \textbf{(H3)} and \textbf{(H4)}. Then Corollary~\ref{COR215} is obvious by Theorem~\ref{THM214}. 

\begin{proof}[Proof of Lemma~\ref{LM214}]
\textbf{(H3')} implies $m$ is equivalent to the Lebesgue measure. Then it follows from \eqref{EQ24MXT} that $\tilde{m}\ll m$, in other words, \textbf{(H3)} is satisfied. On the other hand, \textbf{(H4')} indicates $\tt'$ is of bounded variation. Note that in the notes after Theorem~\ref{THM214}, we already proved $d\tt_*(\tt')\ll \tilde{m}$. Thus we also have $d\tt_*(\tt')\ll m$ and so \textbf{(H4)} is checked. That completes the proof.
\end{proof}

\subsection{Multi-dimensional cases}

In this section, we shall consider the Fukushima's decomposition with respect to the energy form outlined in Lemma~\ref{LM28}. The notations of this section are inherited from \S\ref{SEC23}. Particularly, we always assume that 
\[
	C_c^\infty(U)\subset \FF
\]
and the MAF $M^{[f]}$ in \eqref{EQ23FXT} is equivalent to a Brownian motion. The second assumption implies the Revuz measure of the sharp bracket $\langle M^{[f^i]}\rangle$ equals $m$, i.e. $\mu_{\langle M^{[f^i]}\rangle}=m$ for each $1\leq i\leq d$. Then it follows from \cite[Theorem~4.3.8]{CF12} that $f^i_*m$ is absolutely continuous with respect to the Lebesgue measure on $\mathbb{R}$, where $f^i_*m$ is the image measure of $m$ via the transform $f^i: U\rightarrow \mathbb{R}$. Furthermore, Proposition~\ref{PRO211} indicates if $N^{[f]}$ is of bounded variation, then we almost have $m\ll \lambda_U$, where $\lambda_U$ is the Lebesgue measure on $U$. Thus it is not sudden to further assume 
\[
	m\ll \lambda_U.
\]
Naturally, write $m(dx)=\rho(x)dx$ for some a.e. strictly positive function $\rho\in L^1_\mathrm{loc}(U)$. 

\begin{proof}[Proof of Proposition~\ref{PRO211}]
Let $\{K_n:n\geq 1\}$ be the common generalized nest of increasing compact sets associated with $\{\mu_{N^{[f^i]}}: 1\leq i\leq d\}$ such that for any $u\in \cup_{n\geq 1}\FF_{b, K_n}$, 
\begin{equation}\label{EQ43EUF}
	\EE(u,f^i)=-\langle \mu_{N^{[f^i]}}, u\rangle. 
\end{equation}
We assert \eqref{EQ43EUF} also holds for any $u\in C_c^\infty(U)$. In fact, fix a function $u\in C_c^\infty(U)$. Let $K:=\text{supp}[u]$, and take a relatively compact open set $G$ such that $K\subset G\subset \bar{G}\subset U$. Note that $\{G\cap K_n: n\geq 1\}$ is a generalized nest of the part Dirichlet form $(\EE^G,\FF^G)$ (Cf. \cite[Theorem~3.3.8]{CF12}). Particularly, $\cup_{n\geq 1}\FF_{G\cap K_n}$ is $\EE_1$-dense in $\FF^G$. Since $u\in \FF^G$, we may take a sequence $\{u_k: k\geq 1\}\subset \cup_{n\geq 1}\FF_{G\cap K_n}$ such that $\{u_k: k\geq 1\}$ is uniformly bounded and $u_k$ is $\EE_1$-convergent to $u$ as $k\rightarrow \infty$. Without loss of generality, by \cite[Theorem~2.1.4]{FOT11}, we may also assume $u_k\rightarrow u$, q.e. as $k\rightarrow \infty$.  Clearly, $u_k\in \cup_{n\geq 1} \FF_{b,K_n}$. Thus
\[
	\EE(u_k, f^i)=-\langle \mu_{N^{[f^i]}}, u_k\rangle.
\]
Since $|\mu_{N^{[f^i]}}|(G)<\infty$, by letting $k\rightarrow \infty$, we may obtain $\EE(u, f^i)=-\langle \mu_{N^{[f^i]}}, u\rangle$. 

Now fix a function $h\in C_c^\infty(U)$ with $\text{supp}[h]=W$. For any $u\in C_c^\infty(U)$ and $1\leq i\leq d$, we have
\[
\begin{aligned}
	-\langle \mu_{N^{[f^i]}}, u\cdot h\rangle&=\EE(u\cdot h, f^i)
		\\&=\frac{1}{2}\int_U \frac{\partial (u\cdot h)}{\partial x_i} m(dx)
		\\&=\frac{1}{2}\left(\int_U \frac{\partial u}{\partial x_i}(x)h(x)m(dx)+\int_U \frac{\partial h}{\partial x_i}(x)u(x)m(dx)\right).
\end{aligned}\]
It follows 
\[
	\int_U \frac{\partial u}{\partial x_i}(x)h(x)m(dx)=-2\langle h\cdot \mu_{N^{[f^i]}},u\rangle-\int_U u(x)\frac{\partial h}{\partial x_i}(x)m(dx),
\]
and hence
\[
	\bigg|\int_U \frac{\partial u}{\partial x_i}(x)h(x)m(dx)\bigg|\leq \left( 2\cdot|h\cdot \mu_{N^{[f^i]}}|(W)+\int_{W}\bigg|\frac{\partial h}{\partial x_i}(x)\bigg|m(dx)\right)\cdot ||u||_u
\]
for any $u\in C_c^\infty(U)$ and $1\leq i\leq d$, where $\|u\|_u$ is the uniform norm of $u$. From \cite[Lemma~7.2.2.1]{BH91}, we may deduce that $h\cdot m\ll \lambda_U$. Since $h$ is arbitrary in $C_c^\infty(U)$, we finally achieve $m\ll \lambda_U$. That completes the proof. 
\end{proof}

\begin{remark}
From the above proof, we may easily find that Proposition~\ref{PRO211}  still holds, if $\mu_{N^{[f^i]}}$ only satisfies that there exists a sequence of increasing relatively compact open sets $G_n$ with $\cup_{n\geq 1}G_n=U$ such that $\mu_{N^{[f^i]}}(G_n)<\infty$. If so, we may deduce that \eqref{EQ43EUF} holds for any $u\in C_c^\infty(G_n)$. Then take any fixed function $h\in C_c^\infty(G_n)$, it also follows $h\cdot m\ll \lambda_U$. 
\end{remark}

Our focus is to characterize the local CAF $N^{[f]}$ of zero energy in \eqref{EQ23FXT}. Recall that $N^{[f]}$ is of bounded variation, if each component $N^{[f^i]}$ is of bounded variation for $1\leq i\leq d$, where $N^{[f]}=\left(N^{[f^1]},\cdots, N^{[f^d]}\right)$. The following lemma gives a sufficient condition to the fact that $N^{[f]}$ is of bounded variation and also an explicit expression of $N^{[f]}$. 

\begin{lemma}\label{LM410}
Let $(\EE,\FF)$ be the energy form in Lemma~\ref{LM28}. Assume $C_c^\infty(U)$ is a special standard core of $(\EE,\FF)$, and $m(dx)=\rho(x)dx$ with $\partial \rho/\partial x_i \in L^1_\mathrm{loc}(U)$ for $1\leq i\leq d$, where $\partial \rho/\partial x_i$ is the weak derivative of $\rho$. Then $N^{[f]}$ in \eqref{EQ23FXT} is of bounded variation, and 
\begin{equation}\label{EQ43NFT}
	N^{[f]}_t=\frac{1}{2}\int_0^t \frac{\nabla \rho }{\rho}\left(X_s \right)ds,\quad t\geq 0. 
\end{equation}
\end{lemma}
\begin{proof}
Fix $i$, take any function $u\in C_c^\infty(U)$, we have
\[
\begin{aligned}
	\EE(f^i,u)&=\frac{1}{2}\int_U \nabla f^i(x)\cdot \nabla u(x)\rho(x)dx \\
		&=\frac{1}{2}\int_U \frac{\partial u}{\partial x_i}(x)\rho(x)dx \\
		&=-\frac{1}{2}\int_U u(x)\frac{\partial \rho}{\partial x_i}(x)dx. 
\end{aligned}\]
It follows from $\partial \rho/\partial x_i \in L^1_\mathrm{loc}(U)$ that $-1/2\cdot \partial \rho/\partial x_i\cdot dx$ is a signed Radon measure on $U$, and clearly charges no set of zero capacity. Thus from \cite[Corollary~5.5.1]{FOT11}, we may deduce that $N^{[f^i]}$ is of bounded variation, and its smooth signed measure is
\[
	\mu_{N^{[f^i]}}(dx)=\frac{1}{2}\frac{\partial \rho}{\partial x_i}(x)dx. 
\]
Furthermore, since $\rho>0$, a.e., we also have
\[
	\mu_{N^{[f^i]}}(dx)=\frac{1}{2}\left(\frac{1}{\rho}\frac{\partial \rho}{\partial x_i}\right)(x) m(dx).
\]
Note that $1/\rho \cdot \partial \rho/\partial x_i\in L^1_\mathrm{loc}(U,m)$. Hence we can obtain 
\[
	N^{[f^i]}_t=\frac{1}{2}\int_0^t \left(\frac{1}{\rho}\frac{\partial \rho}{\partial x_i}\right)(X_s)ds,\quad t\geq 0. 
\]
Therefore, $N^{[f]}$ may be written as
\[
N^{[f]}_t=\frac{1}{2}\int_0^t \frac{\nabla \rho }{\rho}\left(X_s \right)ds,\quad t\geq 0. 
\]
That completes the proof.
\end{proof}

Now we take to prove the main result of this section, i.e. Theorem~\ref{THM29} in \S\ref{SEC23}. Clearly, the smooth density function $\rho$ in Theorem~\ref{THM29} satisfies the assumption of Lemma~\ref{LM410}. Particularly, if any of the equivalent assertions in Theorem~\ref{THM29} holds, then $N^{[f]}$ in \eqref{EQ23FXT} could be written as \eqref{EQ43NFT}. Furthermore, denote
\[
	\bar{\EE}(u,v):=\EE(u,v)=\frac{1}{2}\int_U \nabla u(x)\cdot \nabla v(x)m(dx)
\]
for any $u,v\in C_c^\infty(U)$. Denote its smallest closed extension in $(\EE,\FF)$ by $(\bar{\EE}, \bar{\FF})$. Clearly, $(\bar{\EE},\bar{\FF})$ is a regular Dirichlet subspace of $(\EE,\FF)$. Particularly, we may also write the Fukushima's decomposition with respect to $(\bar{\EE},\bar{\FF})$: 
\[
	f(\bar{X}_t)-f(\bar{X}_0)=\bar{M}^{[f]}_t+\bar{N}^{[f]}_t,\quad t\geq 0,
\]
where $(\bar{X}_t)_{t\geq 0}$ is the associated diffusion process of $(\bar{\EE},\bar{\FF})$. Since $(\bar{\EE},\bar{\FF})$ satisfies all the assumptions in Lemma~\ref{LM410}, we can obtain $\bar{N}^{[f]}$ is of bounded variation (no matter whether $N^{[f]}$ is or is not). 

\begin{proof}[Proof of Theorem~\ref{THM29}]
The fact $\textbf{(1)}\Rightarrow \textbf{(2)}$ is obvious from Lemma~\ref{LM410}.
Now, we assume \textbf{(2)} holds. From the proof of Proposition~\ref{PRO211}, we know that for any function $u\in C_c^\infty(U)$, 
\[
	\EE(u,f^i)=-\langle \mu_{N^{[f^i]}}, u\rangle. 
\]
On the other hand, for any $u\in C_c^\infty(U)$, 
	\[
		\mathcal{E}(u,f^i)=\frac{1}{2}\int_U \frac{\partial u}{\partial x_i}(x)\rho(x)dx=-\frac{1}{2}\int_U\frac{\partial \rho}{\partial x_i}(x)u(x)dx.
	\]
Then we can conclude
\[
	 \mu_{N^{[f^i]}}=\frac{1}{2}\frac{\partial \rho}{\partial x_i}dx=\frac{1}{2\rho}\frac{\partial \rho}{\partial x_i}\cdot m.
\]
Particularly, there exists a generalized nest $\{K_n:n\geq 1\}$ of increasing compact subsets of $U$ such that
\begin{equation}\label{MUFI}
\mathcal{E}(u,f^i)=-\frac{1}{2}\int_U u(x)\frac{\partial \rho}{\partial x_i} dx,\quad u\in \bigcup_{n\geq 1}\mathcal{F}_{b,K_n}
\end{equation}
Note that $\cup_{n\geq 1}\FF_{b,K_n}$ is $\mathcal{E}_1^{\frac{1}{2}}$-dense in $\mathcal{F}$. 
Fix $n$ and a function $v\in \mathcal{F}_{b,K_n}$, let us compute the energy measure $\mu_{\langle f^i,v\rangle}$ (Cf. \cite[\S5.2]{FOT11}. For any $g\in C_c^\infty(U)$, it follows from \cite[Theorem~5.6.2]{FOT11} and $\mu_{\langle f^i, f^j\rangle}=\delta_{ij}\cdot m$  that
\[
	gd\mu_{\langle f^i, v\rangle}=d\mu_{\langle f^i,g\cdot v\rangle}-vd\mu_{\langle f^i,g\rangle}
\]
and 
\[
	\mu_{\langle f^i,g\rangle}=\frac{\partial g}{\partial x_i}\cdot m.
\]
Hence from \eqref{MUFI}, we can deduce that
\begin{equation}\label{IUGM}
\begin{aligned}
	\int_U gd\mu_{\langle f^i, v\rangle}&= 2\mathcal{E}(f^i, g\cdot v)-\int_U v(x)\frac{\partial g}{\partial x_i}(x) \rho(x) dx\\
		&=-\int_U g(x) v(x)\frac{\partial \rho}{\partial x_i}(x) dx-\int_U v(x)\frac{\partial g}{\partial x_i}(x) \rho(x) dx\\
		&=-\int_U v(x)\frac{\partial (g\cdot \rho)}{\partial x_i}(x)dx.
\end{aligned}\end{equation}
Now for any $h \in C_c^\infty(U)$, clearly $\partial h/\partial x_i\in C_c^\infty(U)$ for each $1\leq i\leq d$. By using \cite[Theorem~5.6.2]{FOT11} again, it follows from \eqref{IUGM} that
\[\begin{aligned}
	\mathcal{E}(h,v)&=\frac{1}{2}\int_Ud\mu_{\langle h,v\rangle}\\
	&=\frac{1}{2}\int_U \sum_{i=1}^d \frac{\partial h}{\partial x_i}d\mu_{\langle f^i,v\rangle}\\
	&=-\frac{1}{2}\sum_{i=1}^d \int_Uv(x)\frac{\partial}{\partial x_i}\left(\rho \frac{\partial h}{\partial x_i}\right)(x)dx\\
	&=(-\mathcal{L}_0h,v)_m,
\end{aligned}\]
where the operator $\mathcal{L}_0$ is defined by 
\begin{equation}\label{MDAC}
\begin{aligned}
	&\mathcal{D}(\mathcal{L}_0)=C_c^\infty(U),\\
	&\mathcal{L}_0h=\frac{1}{2}\sum_{i=1}^d\frac{1}{\rho}\frac{\partial}{\partial x_i}\left(\rho\frac{\partial u}{\partial x_i}\right),\quad h\in \mathcal{D}(\mathcal{L}_0).
\end{aligned}
\end{equation}
Since $\cup_{n\geq 1}\FF_{b,K_n}$ is $\mathcal{E}_1^{\frac{1}{2}}$-dense in $\mathcal{F}$, we can deduce that for any fixed $h\in C_c^\infty(U)$,
\[
	\mathcal{E}(h,v)=(-\mathcal{L}_0u,v)_m,\quad v\in \mathcal{F}.
\]
Hence $h\in \mathcal{D}(\mathcal{L})$ and $\mathcal{L}h=\mathcal{L}_0h$. That indicates \textbf{(3)}.

Finally, assume \textbf{(3)} holds, i.e. $C_c^\infty(U)$ is a subset of  $\mathcal{D}(\mathcal{L})$. For any $u,v\in C_c^\infty(U)\subset \mathcal{D}(\mathcal{L})\cap \bar{\mathcal{F}}$, we have
\begin{equation}\label{MEUVB}
	\begin{aligned}
		\bar{\mathcal{E}}(u,v)&=\frac{1}{2}\sum_{i=1}^d\int_U \frac{\partial u}{\partial x_i}(x)\frac{\partial v}{\partial x_i}(x)\rho(x)dx\\
			&=\frac{1}{2}\sum_{i=1}^d\int_U-\left(\frac{1}{\rho}\frac{\partial}{\partial x_i}\left(\rho\frac{\partial u}{\partial x_i}\right)\right)(x)v(x)\rho(x)dx\\
			&=\left(-\frac{1}{2}\sum_{i=1}^d\frac{1}{\rho}\frac{\partial}{\partial x_i}\left(\rho\frac{\partial u}{\partial x_i}\right),v\right)_m
			\\&=(-\mathcal{L}_0u,v)_m. 
	\end{aligned}\end{equation}
That implies
\[
	(-\mathcal{L}u,v)_m=\mathcal{E}(u,v)=\bar{\mathcal{E}}(u,v)=(-\mathcal{L}_0u,v)_m.
\]
Since $C_c^\infty(U)$ is dense in $L^2(U,m)$, it follows that $\mathcal{L}u=\mathcal{L}_0u$ for any $u\in C_c^\infty(U)$. That indicates $\mathcal{L}$ is a Markovian self-adjoint extension (Cf.  \cite[\S 3.3]{FOT11}) of $\mathcal{L}_0$. Then from \cite[Theorem~3.3.1]{FOT11}, we know that
\begin{equation}\label{MFSU}
	\mathcal{F}\subset \left\{u\in L^2(U,m):\sum_{i=1}^d\int_U \left(\frac{\partial u}{\partial x_i}\right)^2m(dx)<\infty\right\}.
\end{equation}
Now fix any function $u\in \mathcal{F}\cap C_c(U)$. Set $K:=\text{supp}[u]$, which is a compact subset of $U$. Take a relatively compact open set $G$ such that $K\subset G\subset \bar{G}\subset U$. Since $\rho$ is continuous and strictly positive, there exist two positive constants $c,C$ such that 
\[
	c\leq \rho(x)\leq C
\]
for any $x\in G$. As a sequel, the $\bar{\mathcal{E}}_1$-norm of $\bar{\mathcal{F}}_G$ is equivalent to the norm of $H_0^1(G)$, where $H^1_0(G)$ is the closure of $C_c^\infty(G)$ in the 1-Sobolev space $H^1(G)$. Particularly, $\bar{\mathcal{F}}_G= H_0^1(G)$. It follows from \eqref{MFSU} and $u\in \mathcal{F}\cap C_c(G)$ that $u\in H^1_0(G)$. Thus there exists a sequence $\{u_n:n\geq 1\}\subset C_c^\infty(G)$ such that $u_n$ is convergent to $u$ in $H_0^1$-norm, hence also in $\bar{\mathcal{E}}_1$-norm. Therefore, we can conclude $u\in \bar{\mathcal{F}}$. Since $(\mathcal{E,F})$ is regular, it follows that $\mathcal{F}\subset \bar{\mathcal{F}}$, which indicates \textbf{(1)}. That completes the proof. 
\end{proof}

\section{Examples}\label{SEC5}

\subsection{One-dimensional examples}
 In this section, we shall give several one-dimensional examples. All of them satisfy the basic assumption \eqref{EQ2CCE}, i.e. $C_c^\infty\subset \FF$. Thus we may write their Fukushima's decompositions with respect to the coordinate function $f$ as \eqref{EQ1XXM}. Particularly, from Example~\ref{EXA51} to \ref{EXA54}, the regularity of $N^{[f]}$'s trajectories will become weaker and weaker. 
 
\begin{example}\label{EXA51}
In this example, we shall reintroduce S. Orey's work \cite{O74}, which illustrates perfect diffusion processes that are part of the particular cases of Theorem~\ref{THM21}, i.e. $C_c^\infty(I)\subset \FF$ and $\ss$ is absolutely continuous. Without loss of generality, let $I=\mathbb{R}$ and $\mathbf{X}=(X_t, \mathbf{P}^x)_{x\in \mathbb{R}}$ an irreducible diffusion process with no killing inside, whose scaling function and speed measure are denoted by $\ss$ and $m$. Naturally, $\tt:=\ss^{-1}$. Denote the associated Dirichlet form on $L^2(\mathbb{R},m)$ of $\mathbf{X}$ by $(\EE,\FF)$. 

By \cite[Theorem~1]{O74}, the absolute continuity of $\mathbf{X}$ with respect to a Brownian motion (see \S\ref{SEC1}) has several equivalent conditions, and we pick out two of them as follows:
\begin{description}
\item[(b1)] the scaling function 
\begin{equation}\label{EQ51SXX}
	\ss(x)= \int_0^x\exp\left\{-2\int_0^yb(z)dz\right\}dy,
\end{equation}
and the speed measure $m(dx)=m_b(x)dx$ with the density function 
\[
	m_b(x)= \int_0^x\exp\left\{2\int_0^yb(z)dz\right\}dy
\]
for some $b\in L^2_\mathrm{loc}(\mathbb{R})$;
\item[(b2)] for some $b\in L^2_\mathrm{loc}(\mathbb{R})$, 
\[
	Y_t[b]:=X_t-\int_0^t b(X_u)du,\quad t\geq 0
\]
is a $(\mathbf{P}^x)_{x\in \mathbb{R}}$-Brownian motion. 
\end{description}
Now we assert if so, $\mathbf{X}$ enjoys the following properties:
\begin{description}
\item[(t1)] $C_c^\infty(\mathbb{R})\subset \FF$; 
\item[(t2)] \textbf{(H4')} is satisfied, and $m=\tilde{m}$;
\item[(t3)] $\tt'$ is strictly positive, and $\tt''\in L^2_\mathrm{loc}(J)$. 
\end{description}
The strict positivity of $\tt'$ indicates the absolute continuity of $\ss$. Particularly, $\mathbf{X}$ satisfies all the assumptions of Corollary~\ref{COR215} besides \textbf{(H2)}, and $\mathfrak{D}$ has a unique element.

Let us give a brief proof to the above facts \textbf{(t1)},  \textbf{(t2)} and \textbf{(t3)}. From \eqref{EQ51SXX}, we know that $\ss$ is absolutely continuous and $\ss'$ is strictly positive. It follows that $\tt$ is also absolutely continuous. Note that $m_b(x)=\ss'(x)^{-1}=\tt'\circ \ss(x)$. Thus $m=\tilde{m}$. Clearly, $\tt'\in L^2_\mathrm{loc}(J)$ is equivalent to the fact $(\tt'\circ \ss)^2d\ss=(\tt'\circ \ss)(x)dx=m(dx)$ is Radon, which is obvious. Hence \textbf{(t1)} is verified by Lemma~\ref{LM21}. Then 
\[
	X_t-X_0=B_t+\int_0^t b(X_u)du,\quad t\geq 0,
\]
where $B_t:=Y_t[b]-Y_0[b]$ is a standard Brownian motion, is actually the Fukushima's decomposition of $\mathbf{X}$ with respect to the coordinate function $f$. It follows from Corollary~\ref{COR44} that \textbf{(H4')} is enjoyed and hence \textbf{(t2)} is also checked. Naturally, that $\tt'$ is strictly positive may be deduced from $\tt'\circ\ss(x)=m_b(x)$. Finally, let us prove $\tt''\in L^2_\mathrm{loc}(J)$. In fact, for any compact subset $K\subset \mathbf{R}$, set $K':=\tt(K)$ which is also compact. Then it follows from Corollary~\ref{COR215} that 
\begin{equation}\label{IKBD}
	4\int_{K'} b(x)^2dx=\int_{K'} \frac{(\tt'')^2}{(\tt')^4}(\ss(x))dx=\int_{K} \frac{(\tt'')^2}{(\tt')^3}dx.
\end{equation}
Since $\tt'$ is continuous and strictly positive, we may deduce that
\[
	4\int_{K'} b(x)^2dx\geq \frac{1}{||\tt'||^3_{\infty, K}}\int_{K} \tt''(x)^2dx.
\]
It follows from $b\in L^2_{\text{loc}}$ that $\tt''\in L^2_{\text{loc}}(J)$. That indicates \textbf{(t3)}. 

On the contrary, if $\mathbf{X}$ satisfies \textbf{(t1)}, \textbf{(t2)} and \textbf{(t3)}, we may also deduce that $\mathbf{X}$ is absolutely continuous with respect to a Brownian motion. Note that \textbf{(t1)} and \textbf{(t2)} implies $\mathbf{X}$ enjoys the Fukushima's decomposition in Corollary~\ref{COR215}, whereas $b$ is only in $L^1_\mathrm{loc}(\mathbb{R},m)$. The strictly positivity of $\tt'$ in \textbf{(t3)} indicates $\mathfrak{D}$ has only one element, and from $\tt''\in L^2_\mathrm{loc}(J)$ and \eqref{IKBD}, we may also obtain $b\in L^2_\mathrm{loc}(\mathbb{R})$. Thus from the equivalent condition \textbf{(b2)}, we know $\mathbf{X}$ is absolutely continuous with respect to a Brownian motion. 

Intuitively, \textbf{(t3)} is an equivalent gap from Corollary~\ref{COR215} to the fact $\mathbf{X}$ is absolutely continuous with respect to a Brownian motion that researched in S. Orey's work \cite{O74}. 
\end{example}

\begin{example}\label{EXA52}
Next, we shall give a class of examples that satisfies the assumptions \textbf{(H2)}, \textbf{(H3')} and \textbf{(H4')} in \S\ref{SEC24}. Since \textbf{(H3')} is not essential, namely, we may only raise a scaling function $\ss$ on $\mathbb{R}$ and its inverse function $\tt$ on $J=\ss(\mathbb{R})$ such that
\begin{description}
\item[(s1)] $\ss$ and $\tt$ are both strictly increasing and continuous; 
\item[(s2)] $\tt\in C^1(J)$ and $\tt'$ is absolutely continuous;
\item[(s3)] $\ss$ is not absolutely continuous. 
\end{description}
\end{example}

Take a generalized Cantor set (Cf. \cite{F99}) $K\subset [0,1]$ such that its Lebesgue measure $\lambda(K)>0$. Note that $K$ is nowhere dense and compact. Define a function $\psi$ on $\mathbb{R}$ by
\[
	\psi(x):=d(x,K)=\inf_{w\in K}|x-w|
\]
Clearly, $\psi$ is continuous, strictly positive on $\mathbb{R}\setminus K$, and the set of its all zero points 
\[
	Z_\psi :=\{x\in \mathbb{R}: \psi(x)=0\}
\]
exactly equals $K$. Moreover, one may easily check that for any $x,y\in \mathbb{R}$, 
\[
	|\psi(x)-\psi(y)|\leq |x-y|. 
\]
That indicates $\psi$ is Lipschitz continuous, hence also absolutely continuous. Now define
\[
	\tt(x):=\int_0^x \psi(u)du, \quad x\in \mathbb{R}. 
\]
Clearly, $\tt$ satisfies \textbf{(s2)}. Since $Z_\psi=K$ is a nowhere dense set, it follows that $\tt$ is strictly increasing. Thus $\tt$ and $\ss:=\tt^{-1}$ are both strictly increasing and continuous. Namely, \textbf{(s1)} also holds. Note that $\ss$ is absolutely continuous if and only all the zero points of $\tt'$ is of zero Lebesgue measure. However, $Z_{\tt'}=Z_\psi=K$ and $\lambda(K)>0$. That indicates $\ss$ is not absolutely continuous.  Hence, \textbf{(s3)} is also verified. 

In particular, in the above example, since $K$ is compact, one may easily compute $\tt(-\infty)=-\infty$ and $\tt(\infty)=\infty$. Thus $\ss$ and $\tt$ are both defined on $\mathbb{R}$. Consequently, the associated diffusion process is recurrent, hence also conservative. Furthermore, the absolutely continuous part $\bar{\ss}$ of $\ss$ also satisfies $\bar{\ss}(-\infty)=-\infty$ and $\bar{\ss}(\infty)=\infty$. So by \cite[Remark~3.5]{LY15-2}, every diffusion process associated with an element in $\mathfrak{D}$ is recurrent and we need not worry about its lifetime. In other words, all the solutions in Theorem~\ref{THM214} never explode. 

\begin{example}\label{EXA53}
 Now we shall introduce a class of examples such that $N^{[f]}$ is of bounded variation, but cannot be written as the form of $\int_0^\cdot b(X_s)ds$. In other words, the smooth signed measure $\mu_{N^{[f]}}$ is not absolutely continuous with respect to the speed measure $m$. 
 
Let $\mathbf{X}$ be a diffusion process on $I$ with the scaling function $\ss$, speed measure $m$ and no killing inside. A skew transform at a fixed point $x_0\in I$ with two parameters $\gamma_1,\gamma_2>0$ is an operation on $\mathbf{X}$ to produce another diffusion process $\hat{\mathbf{X}}$, whose scaling function $\hat{\ss}$ and speed measure $\hat{m}$ are given by 
		\[
		\begin{aligned}
			d\hat{\ss}:=\gamma_1 d\ss|_{\{x\in I: x< x_0\}}+\gamma_2 d\ss|_{\{x\in I: x\geq x_0\}},\\
			\hat{m}:=\frac{1}{\gamma_1}m|_{\{x\in I: x< x_0\}}+\frac{1}{\gamma_2}m|_{\{x\in I: x\geq x_0\}}.
		\end{aligned}\]
Let $\tt:=\ss^{-1}$. Set $\ss(x_0)=0$. Then the inverse function $\hat{\tt}$ of $\hat{\ss}$ may be written as
\begin{equation}\label{EQ51TYT}
			\hat{\tt}(y)=\tt\left(\frac{y}{\gamma_1} \right) 1_{\{y<0\}}+ \tt\left(\frac{y}{\gamma_2} \right) 1_{\{y\geq 0\}}, \quad y\in \mathbb{R}. 
\end{equation}
We assert the following basic assumptions remain under the skew transforms:
\begin{description}
\item[(c1)] $C_c^\infty(I)\subset \FF$;
\item[(c2)] $M^{[f]}$ is equivalent to a Brownian motion;
\item[(c3)] $N^{[f]}$ is of bounded variation. 
\end{description}
This fact is easy to check by the equivalent characterizations via $\ss, \tt$ and $m$, say Lemma~\ref{LM21}, Lemma~\ref{LM41} and Proposition~\ref{PRO42}. For example, if $\mathbf{X}$ satisfies \textbf{(c1)}, equivalently, $\tt$ is absolutely continuous and $\tt'\in L^2_\mathrm{loc}$, then it follows from \eqref{EQ51TYT} that $\hat{\tt}$ is also absolutely continuous and $\hat{\tt}'\in L^2_\mathrm{loc}$. Thus $\hat{X}$ also satisfies \textbf{(c1)}. The other two can also be verified through the similar way. 

Particularly, let $\mathbf{X}$ be the Brownian motion on $\mathbb{R}$, and $\gamma_1:=1/\alpha$, $\gamma_2:=1/(1-\alpha)$ for some constant $\alpha$ such that $0<\alpha<1$ and $\alpha\neq 1/2$ ($\alpha=1/2$ corresponds to the Brownian motion). Then the associated diffusion process $\hat{\mathbf{X}}$ after the skew transform is called the $\alpha$-skew Brownian motion. This conception was first raised in \cite{IM74} as a natural generalization of Brownian motion, which behaves like a Brownian motion except that the sign of each excursion at $x_0$ is determined by another independent Bernoulli random variable with parameter $\alpha$. We also refer its complete description to \cite{H81}.

Since Brownian motion clearly satisfies all the three assumptions above, it follows that so does the $\alpha$-skew Brownian motion. Thus  the Fukushima's decomposition of $\alpha$-skew Brownian motion $\hat{\mathbf{X}}$ may be written as
\[
\hat{X}_t-\hat{X}_0=B_t+\hat{N}^{[f]}_t,\quad t\geq 0,
\]
where $B=(B_t)_{t\geq 0}$ is a standard Brownian motion, and $\hat{N}^{[f]}$ is of bounded variation. Since $\alpha\neq 1/2$, we know $\hat{\mathbf{X}}$ is not a Brownian motion, and $\hat{N}^{[f]}$ never disappears. On the other hand, the smooth signed measure $\mu_{\hat{N}^{[f]}}$ is not absolutely continuous with respect to $m$. In fact, from $\tt(x)=x$ and \eqref{EQ51TYT}, we may deduce that
\begin{equation}\label{EQ51TYA}
\hat{\tt}'(y)=\alpha \cdot 1_{\{y<0\}} +(1-\alpha)\cdot 1_{\{y\geq 0\}}, \quad y\in \mathbb{R}. 
\end{equation}
Clearly, $\hat{\tt}'$ is only of bounded variation, but not absolutely continuous. Then Corollary~\ref{COR44} implies $\mu_{\hat{N}^{[f]}}$ is not absolutely continuous with respect to $m$.

It is well known that $\hat{N}^{[f]}$ equals 
\[
	\hat{N}^{[f]}_t=\frac{1}{2}(1-2\alpha)\cdot L^{x_0}_t(\hat{\mathbf{X}}), \quad t\geq 0,
\]
where $\left(L^{x_0}_t(\hat{\mathbf{X}})\right)_{t\geq 0}$ is the local time of $\hat{\mathbf{X}}$ at $x_0$, or equivalently, the PCAF associated with $\delta_{x_0}$. This fact is also obvious by Proposition~\ref{PRO42} and \eqref{EQ51TYA}. 
\end{example}

\begin{example}\label{EXA54}
Finally, we shall give an example such that $N^{[f]}$ is not of bounded variation. Without loss of generality, let $I=\mathbb{R}$. Take the standard Cantor function (Cf. \cite{F99}) $c$ and treat it as a function on $\mathbb{R}$ by defining $c(x)=0$ if $x<0$ and $c(x)=1$ if $x>1$. Let
\[
	\ss(x):=x+c(x),\quad x\in \mathbb{R}. 
\]
Further let the speed measure $m$ be the Lebesgue measure. The associated diffusion process and Dirichlet form of $\ss$ and $m$ are still denoted by $\mathbf{X}$ and $(\EE,\FF)$. Note that the associated Dirichlet form $(1/2\cdot\mathbf{D}, H^1(\mathbb{R}))$ of the Brownian motion is exactly a proper regular Dirichlet subspace of $(\EE,\FF)$, since $dx/d\ss=0$ or $1$, $d\ss$-a.e.  

We assert $\mathbf{X}$ satisfies the assumptions \textbf{(c1)} and \textbf{(c2)} in Example~\ref{EXA53}, but $N^{[f]}$ is not of bounded variation. In fact, one may easily find $\tt:=\ss^{-1}$ is absolutely continuous, and 
\[
	\tt'=0\text{ or }1,\quad \text{a.e.}
\]
We also refer this fact to \cite{FFY05}. 
Moreover, the Lebesgue measure of the set of all the zero points of $\tt'$, i.e. 
\[
	Z_{\tt'}:=\{x\in \mathbb{R}: \tt'(x)=0\}
\]
is positive. Clearly, $\tt'\in L^2_\mathrm{loc}$. Thus Lemma~\ref{LM21} implies \textbf{(c1)}. For \textbf{(c2)}, it suffices to prove $\tt'\circ \ss=1$, a.e. In fact, ($m$ is the Lebesgue measure)
\[
m\left(\{x\in \mathbb{R}: \tt'\circ \ss(x)=0\}\right) =d\tt(Z_{\tt'})=\int_{Z_{\tt'}} \tt'(y)dy=0. 
\]
Finally, no a.e. version of $\tt'$ is of bounded variation. Indeed, it follows from $\tt$ is strictly increasing that $Z_{\tt'}^c$ is measure-dense in $\mathbb{R}$. More precisely, for any open interval $(a,b)$,
\[
	m\left((a,b)\bigcap Z_{\tt'}^c \right)=d\tt\left((a,b)\right)=\tt(b)-\tt(a)>0. 
\]
Particularly, for any $x\in Z_{\tt'}$, any neighbourhood of $x$ contains some points in $Z_{\tt'}^c$.  Thus $\tt'$ is not continuous at any point in $Z_{\tt'}$, whereas $m(Z_{\tt'})>0$. That indicates no a.e. version of $\tt'$ is of bounded variation, since the discontinuous points of a function of bounded variation must be countable. Therefore, from Proposition~\ref{PRO42}, we may obtain that $N^{[f]}$ is not of bounded variation. 

Note that in this example, $\ss$ is not absolutely continuous, thus the set $\mathfrak{D}$ of Dirichlet forms in \S\ref{SEC242} has uncountable elements. Particularly, the associated Dirichlet form $(1/2\cdot\mathbf{D}, H^1(\mathbb{R}))$ of the Brownian motion is the smallest one in it.  Except for the Brownian motion, any other element in $\mathfrak{D}$ still satisfies \textbf{(c1)} and \textbf{(c2)} in Example~\ref{EXA53}, but the zero energy part of its Fukushima's decomposition is not of bounded variation either. 
\end{example}

\subsection{Multi-dimensional examples}

In this section, we shall give two examples. The first one enjoys all the assumptions of Theorem~\ref{THM29}, whereas $C_c^\infty(U)$ is not dense in $\FF$ with the norm $\|\cdot\|_{\EE_1}$. Equivalently, $N^{[f]}$ is not of bounded variation, or $C_c^\infty(U)\not\subset \mathcal{D}(\mathcal{L})$. The second one raises a Dirichlet form $(\EE,\FF)$ on $L^2(\mathbb{R}^d,m)$, where $m$ possesses a strictly positive $C^\infty$-density function $\rho$ on $\mathbb{R}^d\setminus \{0\}$, but $\rho$ explodes at $\{0\}$. The smooth function class $C_c^\infty(\mathbb{R}^d)$ is a special standard core of $(\EE,\FF)$, whereas $N^{[f]}$ is not of bounded variation. 

\begin{example}\label{EXA55}
Let $\mathbf{X}^1,\cdots, \mathbf{X}^d$ be the independent copies of diffusion process that given in Example~\ref{EXA54}, and $\mathbf{X}:=(\mathbf{X}^1,\cdots, \mathbf{X}^d)$ the Cartesian product of them. Clearly, $\mathbf{X}$ is $\lambda_d$-symmetric, where $\lambda_d$ is the Lebesgue measure on $\mathbb{R}^d$. Denote its Dirichlet form on $L^2(\mathbb{R}^d)$ by $(\EE,\FF)$. From Lemma~\ref{LM24} and Theorem~\ref{THM25}, we may easily deduce that  $C_c^\infty(\mathbb{R}^d)\subset \FF$, but it is not dense in $\FF$ with the norm $\|\cdot \|_{\EE_1}$, since $\ss$ given in Example~\ref{EXA54} is not absolutely continuous. On the other hand, $(1/2\cdot \mathbf{D}, H^1(\mathbf{R}^d))$ is naturally a regular Dirichlet subspace of $(\EE,\FF)$. Thus for any $u,v\in C_c^\infty(\mathbb{R}^d)$, 
\[
	\EE(u,v)=\frac{1}{2}\mathbf{D}(u,v)=\frac{1}{2}\int_{\mathbb{R}^d} \nabla u(x)\cdot \nabla v(x) dx. 
\] 
That indicates $(\EE,\FF)$ enjoys all the assumptions of Theorem~\ref{THM29}, but $C_c^\infty(\mathbb{R}^d)$ is not its special standard core. 

For any domain $U$ of $\mathbb{R}^d$, without loss of generality, we may assume a rectangle $I_1\times \cdots \times I_d\subset U$, where $I_1,\cdots, I_d$ are open intervals, and $I_i\cap (0,1)\neq \emptyset$ for some $i$. Then the part Dirichlet form $(\EE^U,\FF^U)$ of $(\EE,\FF)$ on $U$ is still a regular Dirichlet form on $L^2(U)$, and the associated Dirichlet form $(1/2\mathbf{D}_U, H^1_0(U))$ of absorbing Brownian motion on $U$ is a proper regular Dirichlet subspace of $(\EE^U,\FF^U)$. We refer some similar discussions to \cite[\S3.4]{LY15}. Particularly, $(\EE^U,\FF^U)$ also satisfies all the assumptions of Theorem~\ref{THM29}, but $C_c^\infty(U)$ is not its special standard core.  

\end{example}

\begin{example}\label{EXA56}
Let $U=\mathbb{R}^3$, and $m(dx)=\rho(x)dx$ with
\[
	\rho(x):=\frac{\gamma}{2\pi}\frac{\mathrm{e}^{-2\gamma |x|}}{|x|^2}
\]
for a fixed constant $\gamma>0$. Note that $m(\mathbb{R}^3)=1$. Define
	\[
		\begin{aligned}
			&\mathcal{F}=\left\{u\in L^2(\mathbb{R}^3,\rho(x)dx):\nabla u\in L^2(\mathbb{R}^3,\rho(x)dx)\right\}	\\
			&\mathcal{E}(u,v)=\frac{1}{2}\int_{\mathbb{R}^3}\nabla u(x)\cdot\nabla v(x)\rho(x)dx,\quad u,v\in \mathcal{F}. 
		\end{aligned}
	\]
It follows that $(\EE,\FF)$ is a Dirichlet form on $L^2(\mathbb{R}^3,m)$, and $C_c^\infty(\mathbb{R}^3)\subset \FF$. Furthermore, its associated diffusion process $\mathbf{X}$ satisfies the following Fukushima's decomposition
\begin{equation}\label{EQ52XTX}
	X_t-X_0=B_t+N^{[f]}_t,\quad t\geq 0
\end{equation}
where $(B_t)_{t\geq 0}$ is a standard Brownian motion. It is worth asserting the following facts:
\begin{description}
\item[(d1)] $C_c^\infty(\mathbb{R}^3)$ is a special standard core of $(\EE,\FF)$;
\item[(d2)] $N^{[f]}$ is not of bounded variation. 
\end{description}

This example is taken from some studies of mathematical physics, such as \cite{AGHH05, AH82, AHS77}. Particularly, the energy form $(\EE,\FF)$ corresponds to a typical solvable model in quantum mechanics. However, we do not find any references to prove the two properties \textbf{(d1)} and \textbf{(d2)} of $(\EE,\FF)$. So we put their proofs into the appendix. 
\end{example}

\section*{Acknowledgement}
The author in the second order would like to thank Professor Zhi-Ming Ma and Professor Jiangang Ying for many helpful suggestions. He also wants to thank Dr. Hui Xiao for letting him know the reference \cite{IW78}.

\begin{appendix}
%\addcontentsline{toc}{APPENDICES}
\section{Proofs of Example~\ref{EXA56}}

We work on the Hilbert space $L^2(\mathbb{R}^3)$ and start with the unbounded operator 
\begin{equation}\label{LFD}
	L:=\frac{1}{2}\Delta=\frac{1}{2}\sum_{i=1}^3 \frac{\partial^2}{\partial x_i^2}
\end{equation}
with the domain
\begin{equation}\label{LFD2}
	\mathcal{D}(L)=C_c^\infty(\mathbb{R}^3\setminus \{0\}).
\end{equation}
Note that $L$ is symmetric but not self-adjoint on $L^2(\mathbb{R}^3)$. 
The following lemma is taken from \cite{L11} that completely characterizes all the self-adjoint extensions of $L$. We also refer some similar considerations to \cite{AGHH05} and \cite{CKMV10}. 

\begin{lemma}\label{LM1}
All the self-adjoint extensions of $L$ to an operator acting on $L^2(\mathbb{R}^3)$ form a one-parameter family $\{\mathcal{L}_\gamma,\gamma\in \mathbb{R}\}$. The spectrum of $\mathcal{L}_\gamma$ is given by
\[
	\begin{aligned}
		\text{spec}(\mathcal{L}_\gamma)&=(-\infty,0]\bigcup \left\{\frac{\gamma^2}{2}\right\},\quad \gamma>0,\\
		\text{spec}(\mathcal{L}_\gamma)&= (-\infty, 0],\quad \gamma\leq 0.
	\end{aligned}
\]
Furthermore, if $\gamma>0$, then $\gamma^2/2$ is a simple eigenvalue of $\mathcal{L}_\gamma$ with the eigenfunction 
\begin{equation}\label{PGF}
	\psi_\gamma(x)=\frac{\sqrt{\gamma}}{\sqrt{2\pi}}\frac{\text{e}^{-\gamma |x|}}{|x|}.
\end{equation}
\end{lemma}

Note that for a fixed constant $\gamma>0$, the density function $\rho$ in Example is the square of $\psi_\gamma$ given by \eqref{PGF}, i.e. $\rho=\psi_\gamma^2$. The following theorem concludes the fact \textbf{(d1)} in Example~\ref{EXA56}. 

\begin{theorem}
Let $(\EE,\FF)$ be defined  in Example~\ref{EXA56}. Then $(\EE,\FF)$ is a regular Dirichlet form on $L^2(\mathbb{R}^3,m)$ with a special standard core $C_c^\infty(\mathbb{R}^3)$. 
\end{theorem}
\begin{proof}
Since $\psi_\gamma$ is smooth and strictly positive on $\mathbb{R}^3\setminus \{0\}$, it follows from \cite[\S3.3]{FOT11} that $(\mathcal{E,F})$ is a Dirichlet form (may be not regular) on $L^2(\mathbb{R}^3\setminus\{0\},m^0)$, where $m^0=m|_{\mathbb{R}^3\setminus \{0\}}$. Hence it is also a Dirichlet form on $L^2(\mathbb{R}^3,m)$. 
On the other hand, we clearly have
\[
	C_c^\infty(\mathbb{R}^3)\subset \mathcal{F}.
\]
Denote the smallest closed extension of $(\mathcal{E}, C_c^\infty(\mathbb{R}^3))$ in $(\mathcal{E,F})$ by $(\bar{\mathcal{E}},\bar{\mathcal{F}})$. Then $(\bar{\mathcal{E}},\bar{\mathcal{F}})$ is a regular Dirichlet form on $L^2(\mathbb{R}^3,m)$. Let $A$ and $\bar{A}$ be the associated generators of the Dirichlet forms $(\mathcal{E,F})$ and $(\bar{\mathcal{E}},\bar{\mathcal{F}})$ respectively. It suffices to prove $A=\bar{A}$.

We first assert
\begin{equation}\label{CCIM}
	C_c^\infty(\mathbb{R}^3\setminus\{0\})\subset \mathcal{D}(A)
\end{equation}
and 
\begin{equation}\label{AUFD}
Au=\frac{1}{2}\Delta u+\frac{\nabla \psi_\gamma}{\psi_\gamma}\cdot \nabla u
\end{equation}
for any $u\in C_c^\infty(\mathbb{R}^3\setminus\{0\})$. In fact, fix a function $u\in C_c^\infty(\mathbb{R}^3\setminus\{0\})$, we can easily deduce that $u\in \mathcal{F}$ and for any $f\in \mathcal{F}$,
\[
\begin{aligned}
	\mathcal{E}(u,f)&=\frac{1}{2}\sum_{i=1}^3\int\frac{\partial u}{\partial x_i}(x)\frac{\partial f}{\partial x_i}(x)\psi_\gamma(x)^2dx \\
			&=-\frac{1}{2}\sum_{i=1}^3\int f(x) \frac{\partial}{\partial x_i}\left(\frac{\partial u}{\partial x_i}\psi_\gamma^2\right)(x)dx \\
			&=-\frac{1}{2}\sum_{i=1}^3\int f(x)\left(\frac{\partial^2 u}{\partial x_i^2}(x)+2 \frac{\partial u}{\partial x_i}(x)\frac{1}{\psi_\gamma(x)}\frac{\partial \psi_\gamma}{\partial x_i}(x)\right)\psi_\gamma^2(x)dx\\
			&=\left(-\frac{1}{2}\Delta u-\frac{\nabla \psi_\gamma}{\psi_\gamma}\cdot \nabla u, f\right)_m.
\end{aligned}
\]
Thus $u\in \mathcal{D}(A)$ and \eqref{AUFD} holds. Through the same method, we can also obtain
\begin{equation}\label{IMDA}
	1\in \mathcal{D}(A),\quad A1=0.
\end{equation}

Next, define an operator $H$ on $L^2(\mathbb{R}^3)$ by 
\begin{equation}\label{MDHU}
\begin{aligned}
	&\mathcal{D}(H)=\{u\in L^2(\mathbb{R}^3): \psi_\gamma^{-1}\cdot u\in \mathcal{D}(A)\},\\
	&Hu=\psi_\gamma \cdot A(\psi_\gamma^{-1}\cdot u),\quad u\in \mathcal{D}(H).
\end{aligned}
\end{equation}
Since $A$ is self-adjoint on $L^2(\mathbb{R}^3,m)$, it follows that $H$ is a self-adjoint operator on $L^2(\mathbb{R}^3)$. Then from \eqref{CCIM} and \eqref{MDHU}, we may deduce that
\[
C_c^\infty(\mathbb{R}^3\setminus\{0\})\subset \mathcal{D}(H),
\]
and for any $u\in C_c^\infty(\mathbb{R}^3\setminus\{0\})$, 
\[
	Hu=\psi_\gamma \cdot A\left(\psi_\gamma^{-1}\cdot u\right)=\frac{1}{2}\Delta u-\frac{\gamma^2}{2}u.
\]
Let $H_\gamma:=H+\gamma^2/2$. Then $H_\gamma$ is a self-adjoint extension of $L$ defined by \eqref{LFD} and \eqref{LFD2}. Furthermore, \eqref{IMDA} implies
\[
	H_\gamma \psi_\gamma=H\psi_\gamma +\frac{\gamma^2}{2}\psi_\gamma=\frac{\gamma^2}{2}\psi_\gamma.
\]
Hence we can obtain $H_\gamma=\mathcal{L}_\gamma$  from Lemma~\ref{LM1}.

Through the similar way, one may easily find that \eqref{CCIM} and  \eqref{IMDA} are also right for the operator $\bar{A}$. Then we may define another operator $\bar{H}$ relative to $\bar{A}$ on $L^2(\mathbb{R}^3)$ via the same way as \eqref{MDHU}. We may also conclude that $\bar{H}_\gamma:=\bar{H}+\gamma^2/2$ is a self-adjoint extension of $L$ and
\[
	\bar{H}_\gamma=\mathcal{L}_\gamma=H_\gamma. 
\]
Particularly, $A=\bar{A}$. That completes the proof. 
\end{proof}

The following lemma asserts that $\{0\}$ is not a polar set with respect to $(\EE,\FF)$. This fact is very useful to prove Theorem~\ref{THMA4}.

\begin{lemma}\label{LMA3}
The $1$-capacity of $\{0\}$ with respect to $(\EE,\FF)$ is positive, i.e. $\text{Cap}(\{0\})>0$. 
\end{lemma}
\begin{proof}
We only give a brief idea to prove this lemma. Let $B_\epsilon:=\{x: |x|<\epsilon\}$ for any $\epsilon>0$. We may formulate the $1$-equilibrium potential $f_\epsilon$ (Cf. \cite[\S2.1]{FOT11}) of $B_\epsilon$ and
\[
		f_\epsilon(x)=\left\lbrace 
			\begin{aligned}
			&1,\quad 0\leq |x|\leq \epsilon,\\ 
			&\exp\{c(|x|-\epsilon)\},\quad |x|>\epsilon,
			\end{aligned}
		\right. 
	\]
	where $c=\gamma-\sqrt{\gamma^2+2}$. 
Then we may compute 
\[
	\text{Cap}(B_\epsilon)=\EE_1(f_\epsilon, f_\epsilon)=1+\frac{c+\gamma c^2}{\gamma-c}\cdot \text{e}^{-2\gamma \epsilon}. 
\]
Finally,
\[
	\text{Cap}(\{0\})=\inf_{\epsilon>0}\text{Cap}(B_\epsilon)=\frac{\gamma+\gamma c^2}{\gamma-c}>0.
\]
That completes the proof. 
\end{proof}

\begin{theorem}\label{THMA4}
The CAF $N^{[f]}$ in \eqref{EQ52XTX} is not of bounded variation. In other words, the diffusion process $\mathbf{X}$ in Example~\ref{EXA56} is not a semimartingale. 
\end{theorem}
\begin{proof}
Let $(\EE^0,\FF^0)$ be the part Dirichlet form of $(\EE,\FF)$ on $\mathbb{R}^3\setminus \{0\}$. From Theorem~\ref{THMA4}, we know that $(\EE^0,\FF^0)$ is regular on $L^2(\mathbb{R}^3\setminus \{0\},m^0)$ with a special standard core $C_c^\infty(\mathbb{R}^3\setminus \{0\})$. Denote its associated diffusion process by $\mathbf{X}^0$. Then it follows from Lemma~\ref{LM28} and Theorem~\ref{THM29} that $\mathbf{X}^0$ enjoys the following Fukushima's decomposition:
\[
	f(X^0_t)-f(X^0_0)=M^{[f],0}_t+N^{[f],0}_t, \quad t\geq 0,
\]
where $M^{[f],0}$ is equivalent to a Brownian motion and $N^{[f],0}$ is of bounded variation. On the other hand, for any function $u\in C_c^\infty(\mathbb{R}^3\setminus \{0\})$, we have
\[
\begin{aligned}
	\mathcal{E}^0(f^i,u)
						&=\frac{1}{2}\int \frac{\partial u}{\partial x_i} \psi_\gamma^2(x)dx   \\
						&=-\frac{1}{2}\int u(x)\frac{\partial \psi_\gamma^2}{\partial x_i}(x)dx  \\
						&=\int u(x)\frac{\gamma |x|+1}{|x|}\frac{x_i}{|x|}m(dx).
\end{aligned}\]
Denote 
\[
	\nu_i(dx)=-\frac{\gamma |x|+1}{|x|}\frac{x_i}{|x|}m(dx). 
\]
Clearly, $\nu_i$ is a Radon measure on $\mathbb{R}^3\setminus \{0\}$. By \cite[Corollary~5.4.1]{FOT11}, we may conclude that $\nu_i$ is the smooth signed measure of $N^{[f^i],0}$, i.e. $\mu_{N^{[f^i],0}}=\nu_i$.

Now assume $N^{[f]}$ is of bounded variation. From \cite[Theorem~5.5.4]{FOT11}, we know that for each $i$ the signed smooth measure $\mu_i$ of $N^{[f^i]}$ satisfies
\begin{equation}\label{EQFIU}
	\mathcal{E}(f^i,u)=-\langle \mu_i,u\rangle
\end{equation}
for any $u\in \mathcal{F}_{b, F_k}$ where $\{F_k: k\geq 1\}$ is a generalized nest relative to $(\EE,\FF)$ associated with $\mu_i$, i.e. $|\mu_i|(F_k)<\infty$ for any $k\geq 1$. Set $F_k^0:=F_k\cap \left(\mathbb{R}^3\setminus \{0\}\right)$ for any $k\geq 1$. Note that $\{F_k^0: k\geq 1\}$ is a generalized nest relative to $(\EE^0,\FF^0)$ associated with $\mu_i^0:=\mu_i|_{\mathbb{R}^3\setminus \{0\}}$. Furthermore, since $\FF^0_{b,F^0_k}\subset \FF_{b,F_k}$, it follows from \eqref{EQFIU} that 
\[
	\EE^0(f^i,u)=-\langle \mu_i^0,u\rangle,\quad u\in \bigcup_{k\geq 1}\FF^0_{b,F^0_k}. 
\]
Thus we may also conclude $\mu_{N^{[f],0}}=\mu^0_i$ by \cite[Theorem~5.4.2]{FOT11}. Particularly,
\begin{equation}\label{MIDF}
\nu_i=\mu_i\text{ on }\mathbb{R}^3\setminus \{0\}.
\end{equation} 
Note that $\nu_i(\{0\})=0$.  Consequently, $\nu_i$ is also a smooth signed measure with respect to $(\EE, \FF)$, and
\[
	|\nu_i|=\frac{\gamma |x|+1}{|x|}\frac{|x_i|}{|x|}m(dx)
	\]
 is smooth (Cf. \cite[\S2.2]{FOT11}). Since $|x|\leq |x_1|+|x_2|+|x_3|$, it follows that
 \[
 		\frac{\gamma |x|+1}{|x|}m(dx)
 \]	
 is also smooth. Then there exists a quasi-continuous and q.e. strictly positive function $g$ such that (Cf. \cite{F01})
 \begin{equation}\label{IGXF}
 		\int g(x)\frac{\gamma |x|+1}{|x|}m(dx)<\infty.
 \end{equation}
 In particular, Lemma~\ref{LMA3} implies $g(0)>0$. On the other hand, let $B_{\epsilon}:=\{x: |x|\leq \epsilon\}$ and $T_\epsilon$ the hitting time of $B_{\epsilon}^c$. Then (Cf. \cite{M85}) 
 \[
 	\mathbf{E}^0(g(X_{T_\epsilon}))\rightarrow g(0)
 \]
 as $\epsilon\rightarrow 0$, whereas $X_{T_\epsilon}$ is uniformly distributed on $\partial B_{\epsilon}:=\{x:|x|=\epsilon\}$ since $\mathbf{X}$ is rotationally invariant. Thus 
 \[
 	\int_{S^2}g(\epsilon, \sigma)d\sigma\rightarrow g(0)
 \]
 as $\epsilon \rightarrow 0$, where $x=(\epsilon, \sigma)$ is the polar coordinates and $d\sigma$ is the surface measure on $S^2$. Then there is a constant $\delta>0$ such that when $\epsilon<\delta$, 
 \[
 \int_{S^2}g(\epsilon, \sigma)d\sigma>\frac{1}{2}g(0).
 \]
 It follows that
 \[\begin{aligned}
 	\int g(x)\frac{\gamma |x|+1}{|x|}m(dx)&=2\pi\gamma\int_0^\infty  \frac{\gamma r+1}{r}\text{e}^{-2\gamma r}dr\int_{S^2} g(r,\sigma)d\sigma  \\
 	&\geq \left(2\pi\gamma\int_0^\delta \frac{\gamma r+1}{r}\text{e}^{-2\gamma r}dr\right)\cdot \frac{1}{2}g(0),
 \end{aligned}\]
 which is infinite and then contradicts  \eqref{IGXF}. That completes the proof.
\end{proof}

\end{appendix}

\end{document}